\documentclass[11pt]{amsart}
\usepackage[margin=1.0in]{geometry}
\usepackage{xcolor}
\usepackage{lmodern}
\usepackage{hyperref, amsmath, amssymb, mathtools}
\usepackage{verbatim}
\usepackage{amsthm}
\usepackage{amsfonts}
\usepackage{amsmath}
\usepackage{mathrsfs}  
\usepackage{comment}
\usepackage{chngcntr}
\usepackage{ctable}
\usepackage{bm}
\newtheorem{thm}{Theorem}
\newtheorem*{thm*}{Theorem}
\newtheorem{lemma}[thm]{Lemma}
\newtheorem{corollary}[thm]{Corollary}
\newtheorem{conjecture}[thm]{Conjecture}
\newtheorem{prop}[thm]{Proposition}
\newtheorem*{prop*}{Proposition}
\theoremstyle{definition}

\newtheorem{defn}[thm]{Definition}
\newtheorem{remark}[thm]{Remark}
\numberwithin{thm}{section}

\newcommand{\F}{\mathbb{F}}
\newcommand{\Z}{\mathbb{Z}}
\newcommand{\C}{\mathbb{C}}
\newcommand{\R}{\mathbb{R}}
\newcommand{\Q}{\mathbb{Q}}
\newcommand{\N}{\mathbb{N}}
\newcommand{\bx}{\mathbf{x}}
\newcommand{\A}{\mathbb{A}}

\newcommand{\SL}{\textrm{SL}}
\newcommand{\GL}{\textrm{GL}}
\newcommand{\SO}{\textrm{SO}}

\newcommand{\Sym}{\textrm{Sym}}
\newcommand{\sing}{\mathfrak{S}}
\newcommand{\calE}{\mathcal{E}}
\newcommand{\height}{\textrm{ht}}

\newcommand{\calB}{\mathcal{B}}
\newcommand{\calC}{\mathcal{C}}
\newcommand{\calG}{\mathcal{G}}
\newcommand{\calR}{\mathcal{R}}
\newcommand{\calO}{\mathcal{O}}
\newcommand{\calF}{\mathcal{F}}
\newcommand{\scrP}{\mathscr{P}}
\newcommand{\Mat}{\textrm{Mat}}
\newcommand{\vol}{\textrm{vol}}
\newcommand{\Pff}{\textrm{Pff}}
\newcommand{\Skew}{\textrm{Skew}}
\newcommand{\std}{\textrm{std}}

\newcommand{\legendre}[2]{\genfrac{(}{)}{}{}{#1}{#2}}

\newcommand{\PGL}{\textrm{PGL}}
\newcommand{\PO}{\textrm{PO}}

\title{Prime number theorems for polynomials from homogeneous dynamics}
\author[G. Kotsovolis]{Giorgos Kotsovolis}
\address{Department of Mathematics, Princeton University, Princeton, NJ 08540}
\email{gk13@princeton.edu}

\author[K. Woo]{Katharine Woo}
\address{Department of Mathematics, Princeton University, Princeton, NJ 08540}
\email{khwoo@princeton.edu}

\date{\today}
\begin{document}

\maketitle

\begin{center}
{\em Dedicated to Peter Sarnak}\\
{\em on the occasion of his 70th birthday}
\end{center}

\begin{abstract}
    We establish a new class of examples of the multivariate Bateman-Horn conjecture by using tools from dynamics. These cases include the determinant polynomial on the space of $n\times n$ matrices, the Pfaffian on the space of skew-symmetric $2n\times 2n$ matrices, and the determinant polynomial on the space of symmetric $n\times n$ matrices.  
    In particular, let $(V,F)$ be any pair among the following: $(\Mat_n, \det)$, $(\Skew_{2n},\Pff)$, and $(\Sym_n, \det).$ We then obtain an asymptotic for  $$\pi_{V,F}(T)= \#\{v\in V: \max(|v_i|)\leq T, F(v) \text{ is prime}\},$$
    that matches the Bateman-Horn prediction. 
    
    The key ingredients of our proof are an asymptotic count for integral points on the level sets of $F$ given by Linnik equidistribution, a geometric approximation of the box by cones, and an upper bound sieve to bound the number of prime values missed by the approximation. In the case of the determinant polynomial on symmetric matrices, we must also use the Siegel mass formula to compute the product of local densities for the main term.

\end{abstract}

\begin{small}
\setcounter{tocdepth}{1}
\tableofcontents
\end{small}

\section{Introduction}
Let $f(x)\in \Z[x]$ be an irreducible monic polynomial. It was conjectured by Bunyakovsky in 1857 that $f(n)$ will be prime for infinitely many $n\in \N$. This conjecture was made quantitative by Bateman and Horn in \cite{BatemanHorn}, and is now known as the Bateman-Horn conjecture:
\begin{conjecture}[The Bateman-Horn conjecture]
    Let $f(x)$ be a nonconstant irreducible monic polynomial over $\Q$. Then as $N\rightarrow\infty:$ $$\#\{n\in [1,N]: f(n) \textrm{ is prime}\} \sim \sing \cdot \int_{0}^N \frac{dx}{\log^+(f(x))},$$
    where we write $$\log^+(y) = \begin{cases}
        \max(2,\log(y)), & y> 0 \\ \infty, & y\leq 0,
    \end{cases}$$
    and where the singular series $\sing$ is defined as a product of local densities: $$\sing = \prod_{p}\left(\frac{1-p^{-1}\#\{x\in \F_p: f(x)\equiv 0 \bmod p\}}{1-1/p}\right).$$
\end{conjecture}

While the Bateman-Horn conjecture has yet to be solved for any nonlinear polynomial, mathematicians have made substantial progress proving prime number theorems for multivariate polynomials. 
In the Appendix of \cite{DS}, Destagnol and Sofos wrote down the expected asymptotic for the multivariate Bateman-Horn.
\begin{conjecture}[The Bateman-Horn conjecture for multivariate homogeneous polynomials]\label{conj: DS Bateman-Horn}
Let $F(\bx)$ be a nonconstant irreducible polynomial over $\Q$ in $n$ variables. For a box $\mathcal{B}\subset \R^n$, we denote $T\mathcal{B} := \{t\bx: t\in [0,T], \bx\in \mathcal{B}\}.$ Then, we have that as $T\rightarrow \infty$, $$\#\{\bx\in T\mathcal{B}: F(\bx) \textrm{ is prime}\} \sim \sing \cdot \int_{T\mathcal{B}} \frac{d\bx}{\log^+( F(\bx))},$$
where the singular series $\sing$ is defined as a product of local densities: $$\sing = \prod_{p} \left(\frac{1-p^{-n} \#\{\bx\in \F_p^n: F(\bx)\equiv 0 \bmod p\}}{1-1/p}\right).$$ 
    
\end{conjecture}

We remark that in the multivariate setting it may be more natural to look at integers that generate prime ideals $(p)\subset \Z$, rather than positive primes. For instance, $$F(x,y) = 1-x^2-y^2$$
is an absolutely irreducible polynomial with no local restrictions that does not take any positive prime values. Additionally, we would like to consider regions more general than boxes. As such, we write down another formulation of the Bateman-Horn conjecture. 
\begin{conjecture}\label{conj: Bateman-Horn}
Let $F(\bx)$ be a nonconstant irreducible homogeneous polynomial over $\Q$ in $n$ variables. For any convex compact region $\mathcal{R}_0\subset \R^n$, we denote $T\calR_0 := \{t\bx: t\in [0,T], \bx\in \mathcal{R}_0\}.$ Then, we have that as $T\rightarrow \infty$, $$\#\{\bx\in T\mathcal{R}_0:(F(\bx))\subset\Z \textrm{ is a prime ideal}\} \sim \sing \cdot \int_{T\mathcal{R}_0} \frac{d\bx}{\log^+(\vert F(\bx)\vert)},$$
where $$\sing = \prod_{p} \left(\frac{1-p^{-n} \#\{\bx\in \F_p^n: F(\bx)\equiv 0 \bmod p\}}{1-1/p}\right).$$
    
\end{conjecture}

In this paper, we will prove the multivariate Bateman-Horn conjecture for a new class of polynomials. Specifically, in Theorem \ref{thm: PNT det Mat}, we prove a prime number theorem for the determinant polynomial on $n\times n$ matrices;  in Theorem \ref{thm: PNT Pff Skew}, we prove a prime number theorem for the Pffafian on $2n\times 2n$ skew symmetric matrices; in Theorem \ref{thm: PNT det Sym}, we prove the more involved prime number theorem for the determinant polynomial on $n\times n$ symmetric matrices. A crucial ingredient in these proofs is the following uniform upper bound (Proposition \ref{prop: upper bound}) for any absolutely irreducible homogeneous polynomial $F$:
$$\limsup_{\calR_0\subset \R^n} \lim_{T\rightarrow\infty} \frac{\#\{\bx\in T\calR_0(\Z): F(\bx)\textrm{ is prime}\}}{|T\calR_0|/\log(T)} \leq c_F.$$

Before introducing these results, we will give a brief survey of previous known results for prime number theorems for multivariate polynomials.

\subsection{History of multivariate Bateman-Horn conjecture}
First, for polynomials $F\in \Z[x_1,...,x_n]$ with $n$ sufficiently large in the degree $d$, the Hardy-Littlewood circle method can produce such asymptotics. In particular, for $n > (d-1)2^{d-1}$, the Bateman-Horn conjecture was proved by Destagnol and Sofos in \cite{DS} using the circle method. If $F$ is diagonal, then an even better bound exists: for $n\geq \lceil d\log(d) + 4.21\rceil$, Br\"udern and Wooley's work \cite{BrudernWooley} on Waring's problem gives the Bateman-Horn conjecture. 

Beyond the circle method, the other known cases of Bateman-Horn occur when there is extra algebraic structure. In \cite{FI1} and \cite{FI2}, Friedlander and Iwaniec establish that there are infinitely many primes of the form $x^2+y^4$. Later, Heath-Brown \cite{HB} solved the case of $x^3+2y^3$; this was then followed by Heath-Brown and Moroz \cite{HBMoroz} who finished off the case of irreducible binary cubic forms. Finally, Maynard \cite{MaynardNormForm} used the structure of norm forms to generalize Heath-Brown and Moroz's result and handled incomplete norm forms of degree $n$ in $\geq 15n/22$ variables. We remark that the results of \cite{FI1}, \cite{FI2}, \cite{HB}, and \cite{HBMoroz} all use the structure of norm forms and the structure of prime factorization in number fields in their respective proofs. We also note that all of these polynomials factor over $\overline{\Q}$, whereas our polynomials will be irreducible over $\overline{\Q}$.

\subsection{Statements of the main theorems}
In this section, we will present prime number theorems for three specific polynomials.

First, we consider the determinant polynomial $\det$ on $n\times n$ integral matrices. This polynomial $\det(x_{11},x_{12},...,x_{nn})$ is a non-diagonal homogeneous polynomial in $n^2$ variables of degree $n$ and is irreducible over $\overline{\Q}.$ For $n\geq 4$, a prime number theorem for this polynomial is beyond what is known from the circle method. However, by utilizing results from dynamics, such as the equidistribution of Hecke orbits, along with an approximation argument, we can achieve the following prime number theorem. 

\begin{thm}\label{thm: PNT det Mat}
    Define the following prime counting function: 
    $$\pi_{\det}(T) := \#\{A\in \Mat_n(\Z): \|A\|\leq T, \det(A) \text{ is prime}\},$$
    where $\|A\|$ denotes the $L^\infty$-norm. As $T\rightarrow\infty$, we have that 
    $$\pi_{\det}(T) = (1+o(1)) \cdot \prod_{j=2}^n \zeta(j)^{-1} \int_{\|X\|\leq T} \frac{1}{\log^+(\det(X))}dX,$$
    where $dX$ is the Euclidean measure on $\Mat_n(\R)\cong \R^{n^2}.$ Moreover, Conjecture \ref{conj: Bateman-Horn} holds for the determinant polynomial. 
\end{thm}

\begin{remark}Our methods lead us to deal with even more general regions, as will be discussed in \S4 and \S5. In particular, if $\calR_0\subset \Mat_n(\R)$ satisfies that for $\bx\in \calR_0$, and $t\in [0,1]$, we have that $t\bx\in \calR_0$, as well as that the boundary of $\calR_0$ has Euclidean measure zero, then we have that $$\#\{A\in \Mat_n(\Z)\cap T\calR_0: \det(A) \text{ is prime}\} \sim \prod_{j=2}^n \zeta(j)^{-1} \int_{T\calR_0} \frac{1}{\log^+(\det(X))}dX.$$ This fact will be a consequence of Theorem \ref{PNT cones Hardy-Littlewood} and Theorem \ref{thm: cones to boxes}. Some examples of such $\calR_0$ include the star-shaped region and compact sets $\calR_0$ defining a sector.  
\end{remark}

We also observe that this result implies that $$\pi_{\det}(T) \sim \frac{1}{2}\cdot \prod_{j=2}^n \zeta(j)^{-1} \cdot \frac{2^{n^2}T^{n^2}}{n\log(T)}.$$
Here the factor of $1/2$ comes from the fact that asymptotically half of the matrices in the box $\|X\|\leq T$ satisfy that $\det(X)>0$.

Second, we consider the Pfaffian polynomial $\Pff$ on skew-symmetric $2n\times 2n$ integral matrices. This polynomial is a nondiagonal homogeneous polynomial in $2n^2-n$ variables of degree $n$ and is irreducible over $\overline{\Q}.$ For $n\geq 5$, a prime number theorem for this polynomial is beyond what the circle method can currently handle. Again, by inputting tools from dynamics (in particular Ratner's theorem), we can establish a prime number theorem for the Pfaffian. 

\begin{thm}\label{thm: PNT Pff Skew}
    Let $n\geq 2$. Define the following prime counting function: 
    $$\pi_{\Pff}(T) := \#\{A\in \Mat_{2n}(\Z): \|A\|\leq T, A^T = -A, \Pff(A) \text{ is prime}\},$$
    where $\|A\|$ denotes the $L^\infty$ norm. As $T\rightarrow\infty$, we have that 
    $$\pi_{\Pff}(T) = (1+o(1))\cdot \prod_{\substack{3\leq n\leq 2n-1\\ j\text{ odd}}} \zeta(j)^{-1} \cdot \int_{\substack{\|X\|\leq T\\ X^T = -X}} \frac{1}{\log^+(\Pff(X))}dX.$$
    Here $dX$ is the Euclidean measure on $\Skew_{2n}(\R) = \{X\in \Mat_{2n}(\R): X^T = -X\} \cong \R^{2n^2-n}.$ Moreover, Conjecture \ref{conj: Bateman-Horn} holds for the Pfaffian.
\end{thm}
\begin{remark}As with the determinant case, we again handle more general regions, as a consequence of Theorem \ref{thm: PNT Pfaffian cones} and Theorem \ref{thm: cones to boxes}.
\end{remark}

Note that both the determinant and the Pfaffian polynomials belong to a more general dynamical framework; this set up will be discussed further in \S\ref{sec: general framework}.

Our final example is the determinant polynomial $\det$ on symmetric $n\times n$ matrices. This polynomial is a nondiagonal homogeneous polynomial in $\frac{n(n+1)}{2}$ variables of degree $n$ and it is irreducible over $\overline{\Q}$; furthermore, the variety $\{X\in \Sym_n(\overline{\Q}):\det(X) = 0\}$ \textit{is not smooth}. For $n=1$, the Bateman-Horn conjecture is answered by the prime number theorem. When $n=2$, the determinant becomes a ternary quadratic form and thus can be handled by the results of Duke in \cite{Duke}. However, for $n\geq 3$, this polynomial is currently beyond what is known (from the circle method or other techniques). The proof of the following theorem makes use of Linnik type equidistribution as well as the Siegel mass formula.

\begin{thm}\label{thm: PNT det Sym}
    Let $n\geq 3$. Define the following prime counting function: 
    $$\pi_{\Sym}(T):= \#\{A\in \Mat_n(\Z): \|A\|\leq T, A^T=A, \det(A) \text{ is prime}\},$$
    where $\|A\|$ denotes the $L^\infty$-norm. As $T\rightarrow\infty$, we have that $$\pi_{\Sym}(T) = (1+o(1)) \cdot \prod_{\substack{3\leq j\leq n\\ j\text{ odd}}}\zeta(j)^{-1} \cdot \int_{\substack{\|X\|\leq T \\ X^T = X}} \frac{1}{\log^+(\det(X))}dX.$$
    Here $dX$ is the Euclidean measure on $\Sym_n(\R) = \{X\in \Mat_n(\R): X^T = X\} \cong \R^{n(n+1)/2}.$ Moreover, Conjecture \ref{conj: Bateman-Horn} holds for the determinant polynomial on symmetric matrices.
\end{thm}
\begin{remark}Again, we can handle more general regions than boxes; see Theorem \ref{thm: cones to boxes} and \S\ref{sec: sym}. 
\end{remark}

Moreover, we prove a more precise result depending on the signature of the symmetric matrices involved; this is discussed further in \S\ref{sec: sym}.  
\begin{remark}
    The question of prime determinants for the determinant polynomial on $\Mat_n(\Z)$ and $\Sym_n(\Z)$ was raised by Kovaleva in her insightful investigation in \cite{Kovaleva}. Kovaleva counts in \cite{Kovaleva} matrices with squarefree, instead of prime, determinants and resolves the variant question of counting $n\times n$ symmetric integer matrices with a squarefree minor of size $n-r$ for $r>0$. 
\end{remark}

\subsection{Outline of the paper}
First, in \S\ref{sec: general framework} we describe the general framework that all of our examples belong to -- they are all almost Hardy-Littlewood systems (Definition \ref{def: Hardy-Littlewood system}). We present a general prime number theorem (\ref{General PNT on cones}) under certain conditions and relate the problem of counting primes to the Linnik problem. Additionally, in \S2 we will discuss the Tamagawa measure, the ``Bateman-Horn measure'', and how these two measures relate to the Euclidean measure. 

In \S3, we prove a prime number theorem for the determinant polynomial on integral matrices and a prime number theorem for the Pfaffian on integral skew-symmetric matrices \textit{for particularly structured regions}; namely, we look at cones. 

In \S\ref{sec: eps cutting}, we start our procedure for turning our prime number theorem on cones into a theorem on boxes (and more general regions). We introduce the notion of an $\epsilon$-cutting, which is an approximation of the desired region by cones. The idea of $\epsilon$-cuttings comes from the principle of Riemannian integration -- one can approximate a general compact region $\calR$ using sequences of cones of smaller and smaller width. 

In \S\ref{sec: sieve}, we introduce the second ingredient of our procedure: an upper bound sieve on the part of the box that is not covered by the approximation by cones. Section \ref{sec: sieve} does not make use of the dynamics framework and it applies to a general homogeneous irreducible polynomial. It yields an upper bound for Conjecture \ref{conj: Bateman-Horn}, which is off by a constant independent of $\calR_0.$ Finally in \S\ref{sec: sieve}, we present the proofs of Theorem \ref{thm: PNT det Mat} and Theorem \ref{thm: PNT Pff Skew}.

In \S\ref{sec: Siegel}, we establish the necessary background on the Siegel mass formula to prove Theorem \ref{thm: PNT det Sym}, which is subsequently proved in \S\ref{sec: sym}. Finally, in \S\ref{sec; remarks}, we comment on variant questions that our methods can handle. 

\section{General framework of invariant polynomials for homogeneous spaces}\label{sec: general framework}
In this section, we aim to put our choices of polynomials in a general framework -- they are the invariant polynomials of homogeneous spaces. 

\subsection{The Linnik Problem} 
The proofs of Theorems \ref{thm: PNT det Mat}, \ref{thm: PNT Pff Skew}, and \ref{thm: PNT det Sym} rely on the resolution of the Linnik problem in the respective set ups. In this subsection, we explain the Linnik problem.

Let $F(\bx)\in \Z[\bx]$ be a general polynomial in $n$ variables, let $$V_m = \{\bx: F(\bx) = m\}$$
be the level sets of $F$, and let $V_m(\Z)$ denote the integer points of $V_m$. The Linnik problem studies the equidistribution of $V_m(\Z)$ when projected onto $V_1(\R)$ as $m\rightarrow\infty$. 
We list some examples. 

Linnik \cite{Linnik} used his \textit{ergodic method} to prove (conditional on $m$ being a quadratic residue for any fixed odd prime) that for $F(\bx) = x_1^2+x_2^2+x_3^2,$ for $\mu$ the normalization of the Euclidean measure on $\mathbf{S}^2$, and for $\Omega\subset \mathbf{S}^2$ a subset with boundary of measure zero, the following holds:
$$\lim_{m\rightarrow\infty} \frac{\#V_m(\Z)\cap m^{\frac{1}{2}}\Omega}{\#V_m(\Z)} = \mu(\Omega).$$ This result was later proved unconditionally by Duke in \cite{Duke} using bounds on the Fourier coefficients of half-weight modular forms proved by Iwaniec in \cite{Iwaniec}. 

In \cite{LinnikSkubenko}, Linnik and Skubenko solve the corresponding  problem for the determinant polynomial on  $\GL_n(\R)$. Notice that unlike the case of the sphere, $V_m(\Z)$ is infinite for every $m\in \mathbb{N}$, which can be understood as a consequence of $\SL_n(\R)$ having infinite measure and hence it is necessary to intersect with some compact set. We refer the reader to Sarnak's 1990 ICM talk \cite{Sarnak} for the connection between this problem and the equidistribution of Hecke orbits. Sarnak shows the following result, which is effective with a power saving: let $\Omega\subset \SL_n(\R)$ be a ``nice'' compact set (as defined in Definition \ref{def: nice}), $F$ be the determinant polynomial, and $\mu$ the Euclidean measure on $\SL_n(\R)$ (induced from $\GL_n(\R)$). Then, as $m\rightarrow\infty:$ 
$$\#V_m(\Z)\cap m^{\frac{1}{n}}\Omega \sim \prod_{j=2}^n \zeta(j)^{-1}\cdot G_n(m)\mu(\Omega),$$ where for $m=\prod_{i=1}^r p_i^{e_i},$ we define $$G_n(m)=\prod_{i=1}^r\frac{(p_i^{e_i+1}-1)(p_i^{e_i+2}-1)\dots(p_i^{e_i+n-1}-1)}{(p_i-1)(p_i^{2}-1)\dots(p_i^{n-1}-1)}.$$ 

A robust approach to these Linnik type problems is to use Ratner's work on measure rigidity \cite{Ratner}. We will make use of the work of Eskin and Oh \cite{EO}.

\subsection{The general framework}
Let $G$ denote some semisimple real algebraic group defined over $\Q$ and let $\rho:G\rightarrow\GL(V)$ be a rational representation of $G$ on a finite vector space $V$. Assume that there is some polynomial $F$ that is invariant under the right action of $\varrho$, i.e. 
$$F(v\varrho(g)) = F(v),  \forall g\in G,$$
and that $F$ is irreducible over $\overline{\Q}$.
We denote the level sets of $F$ by $$V_m := \{ v\in V: F(v) = m\}.$$ 

For some $v_0\in V$, let $H=\text{stab}_{v_0}$ denote the stabilizer of $v_0$ inside $G$. We require that the set $(V,G,H,F)$ satisfies the following strong assumptions:
 
\begin{itemize}
    \item $V_1 = v_0 G$ for some vector $v_0\in V(\R)$, i.e. $V_1$ can be identified with the homogeneous space $$V_1\cong G/H,$$
    \item for each $m$, there is a $\lambda_m\in \R$ such that $V_m =\lambda_m V_1$,
    \item $V_1(\Q)\neq \emptyset$, i.e. $V_1$ contains a rational point,
    \item $H(\R)\subset G(\R)$ is a connected, simply connected, and semisimple subgroup. We also assume that it is a proper maximal subgroup of $G(\R)$ and \textit{contains no compact factors},
    \item if $m_0\in \Z$ satisfies that $\#\{v\in V(\Z): F(v) = m_0\}\neq 0$, then we must have $$\#\{m\in\Z: \lambda_m^{-1} V_m(\Z) = \lambda_{m_0}^{-1} V_{m_0}(\Z)\} <\infty.$$
\end{itemize} 
When $V_1\cong G/H$ is an affine symmetric space, these varieties satisfy the Hasse principle, as was shown by Borovoi in \cite{Borovoi}; in more generality we are assuming the existence of an initial rational point. 
The fact that the stabilizer $H(\R)$ is connected and semisimple ensures the existence and uniqueness of the Tamagawa measure on the level sets of $F$. The final two conditions are what is necessary in order to apply Theorem 1.9 of Oh \cite{Oh}.
\begin{remark}
    We note that the case $V_1 = \bigcup_i  v_i G$ is a finite union of orbits, where the stabilizers of $v_i$ satisfy all of the above conditions, is essentially the same and we discuss the single orbit case here only to make the exposition simpler. 
\end{remark}

\begin{defn}\label{def: nice}
We will call a compact set $\Omega\subset V_1(\R)$ \textbf{nice} if its boundary has measure zero on $V_1(\R).$
For a nice compact subset $\Omega\subset V_1(\R)$, we define the \textbf{cone of} \bm{$\Omega$} \textbf{of height} \bm{$T$} as the following region: $$T\Omega := \{v= t\omega: t\in [0,T], \omega\in \Omega\}.$$ 
\end{defn}

Under the assumptions listed above, Eskin and Oh prove in \cite{EO} that for any $\Omega$ nice compact subset of $V_1(\R)$ we have $$\#V_m(\Z)\cap m^{\frac{1}{n}}\Omega\sim \omega_m \mu(\Omega),\,m\rightarrow\infty,$$ where $\mu$ is some $G-$invariant Borel measure on $V_1$ and $\omega_m$ is defined as follows: for any arithmetic lattice $\Gamma$ in $G(\Q)$ with $V_m(\Z)\Gamma\subset V_m(\Z)$ for all $m\in\Z$, we have $$\omega_m(\Gamma)=\frac{\sum\limits_{\xi\Gamma \subset V_m(\Z)}\mu_H((H\cap g_{\xi}^{-1}\Gamma g_{\xi})\backslash H)}{\mu_G(\Gamma\backslash G)},$$ where 
\begin{itemize}
    \item the sum is over all disjoint orbits $\xi\Gamma$ inside $V_m(\Z)$; this sum is finite as a consequence of the theorem of Borel and Harish-Chandra \cite{BH},
    \item $g_{\xi}$ is a point of $G$ such that $v_0=(m^{-\frac{1}{n}}\xi )g_{\xi}$,
    \item the Borel measure $\mu$ and the Haar measures $\mu_G,\mu_H$ are compatible in the sense of Weil \cite{We}. 
    
    \end{itemize}
    
In \cite{Oh}, Oh proves that $\omega_m$ does not depend on the choice of $\Gamma$ and thus the notation is consistent. An immediate corollary is that for nice compact sets $\Omega_1,\Omega_2\subset V_1$: $$\frac{\#V_m(\Z)\cap m^{\frac{1}{n}}\Omega_1}{\#V_m(\Z)\cap m^{\frac{1}{n}}\Omega_2}\sim\frac{\mu(\Omega_1)}{\mu({\Omega_2})}.$$ Finally, Oh computes the constant $\omega_m$ in terms of local masses and derives the following, which will be crucial for our analysis:
\begin{thm}[Oh, \cite{Oh}] \label{Oh's Theorem}
Let $V$ be a linear space with a given $\Z$-structure satisfying the above conditions. Then for any nice compact set $\Omega \subset V_1(\R)$, as $m\rightarrow\infty$, 
$$\#V_m(\Z) \cap \calR(m^{1/\deg(F)},\Omega) =\mu_\infty(\Omega)\cdot m^{\dim(V)/\deg(F)-1} \cdot \sing_m \cdot (1+o_{\Omega}(1)),$$
    where the singular series is given by $$\sing_m = \prod_{q} \lim_{k\rightarrow\infty} \frac{V_m(\Z/q^k\Z)}{q^{k\dim(V_m)}},$$
    and the dependence on the compact set is given by $$\mu_\infty(\Omega) = \lim_{\epsilon\rightarrow 0}\frac{\vol(\{v\in [1-\epsilon,1+\epsilon]\Omega\})}{2\epsilon},$$ as computed by Borovoi and Rudnick in \cite{BR}. 
\end{thm}

This result of Oh implies that under the conditions above on $(V,G,H,F)$, we can conclude that $(V,F)$ satisfies the predictions made by Hardy and Littlewood for rational points on varieties; in particular, we will call these systems Hardy-Littlewood systems. 
\begin{defn}\label{def: Hardy-Littlewood system}
   A linear space $V$ with polynomial $F$ is a \textbf{Hardy-Littlewood system} if it satisfies that as $m\rightarrow \infty$, $$N(m,\Omega) = \mu_\infty(\Omega)\cdot m^{\dim(V)/\deg(F)-1} \cdot \sing_m \cdot (1+o_{\Omega}(1)),$$
    where the singular series is given by $$\sing_m = \prod_{q} \lim_{k\rightarrow\infty} \frac{\#V_m(\Z/q^k\Z)}{q^{k\dim(V_m)}},$$
    and the dependence on the compact set is given by $$\mu_\infty(\Omega) = \lim_{\epsilon\rightarrow 0}\frac{\vol(\{v\in [1-\epsilon,1+\epsilon]\Omega\})}{2\epsilon}.$$ 
    Here the volume denoted is the Euclidean volume.
\end{defn}

The above asymptotic result on the count of integral points on $V_m$ will allow us to derive a prime number theorem on cones by summing over the level sets with prime values. In particular, if $(V,F)$ is a Hardy-Littlewood system and the singular series $\sing_m$ satisfies that for all sufficiently large primes $p$, $$\sing_p = \sing_{V,F} (1+O(p^{-\delta})),$$
where $\sing_{V,F}$ is the singular series from Conjecture \ref{conj: DS Bateman-Horn} and $\delta>0$ is a positive constant, then the following prime number theorem holds for any nice compact set $\Omega\subset V_1(\R)$: 
\begin{align}\#\{v\in V(\Z)\cap T\Omega: F(v)\text{ is prime}\}\sim& \sing_{V,F} \cdot \int_{T\Omega} \frac{1}{\log^+(F(v))}dv. \label{General PNT on cones} \end{align}
The proof of this fact follows exactly the proof presented in \S\ref{sec: PNT cones} for the determinant polynomial and the Pfaffian.

\subsection{Relating the measures}

So far, we have discussed the measure $\mu_{\infty}$ on $V_1$, which appears in Theorem \ref{Oh's Theorem}, and the Euclidean measure, which appears as the archimedean part of the Bateman-Horn conjecture. In this subsection, we wish to provide a relation between these two measures. In order to establish this connection, we utilize the co-area formula.

\begin{lemma}\label{lem: the real part for cones PNT}
    For $\Omega$ a nice compact connected subset of $V_1(\R)$, we have that as $T\rightarrow\infty$, $$\int_{T\Omega} \frac{1}{\log^+(F(x))}dx = \mu_\infty(\Omega) \cdot \frac{T^{\dim(V)}}{\dim(V)\log(T)} (1+o(1)).$$ 
\end{lemma}
\begin{proof}
From the computations of Borovoi and Rudnick in \cite{BR}, we know that $$\mu_\infty(\Omega) = \lim_{\epsilon\rightarrow 0} \frac{1}{2\epsilon}\int_{[1-\epsilon,1+\epsilon]\Omega} dx,$$
where $dx$ is the Euclidean measure on $V(\R).$ The co-area formula then allows us to rewrite this formula in terms of an integral over $\Omega$: \begin{align*}
    \int_{[1-\epsilon,1+\epsilon]\Omega} dx 
    &= \int_{1-\epsilon}^{1+\epsilon} \int_{r^{1/\deg(F)} \Omega}|\nabla F(y)|^{-1} dy dr\\ 
    &= \int_{1-\epsilon}^{1+\epsilon} r^{\frac{\dim(V)}{\deg(F)}-1}\int_{\Omega} |\nabla F(y)|^{-1} dy dr \\
    &=\left(\int_{\Omega} |\nabla F(y)|^{-1} dy\right) \cdot ((1+\epsilon)^{\dim(V)/\deg(F)}-(1-\epsilon)^{\dim(V)/\deg(F)})\cdot \frac{\deg(F)}{\dim(V)}.
\end{align*}
Hence, taking the limit as $\epsilon\rightarrow 0$, we have that $$\mu_\infty(\Omega) = \int_{\Omega} |\nabla F(y)|^{-1} dy. $$

On the other hand, we consider the archimedean part of the Bateman-Horn conjecture. Applying the co-area formula again, we have that as $T\rightarrow\infty$:
\begin{align*}
    \int_{T\Omega}\frac{1}{\log^+(F(x))} dx
    &= \int_2^{T^{\deg(F)}} \frac{1}{\log(r)} \int_{r^{1/\deg(F)}\Omega} |\nabla F(y)|^{-1} dy dr \\
    &= \int_{2}^{T^{\deg(F)}} \frac{r^{\frac{\dim(V)}{\deg(F)}-1}}{\log(r)} \int_{\Omega} |\nabla F(y)|^{-1} dy dr \\
    &= \left(\int_{\Omega} |\nabla F(y)|^{-1} dy\right) \cdot \frac{T^{\dim(V)}}{\dim(V)\log(T)}(1+o(1)).
\end{align*}
\end{proof}
\begin{remark}
We recall that in Oh \cite{Oh}, the measure $\mu_\infty$ is a gauge form, a nowhere zero regular differential form of maximal degree, which is used to construct the unique Tamagawa measure on $V_1(\mathbb{A})$. 
Hence, the co-area formula gives us a precise expression for this Gauge form $d\mu_{\infty}$ on $V_1:$ $$d\mu_{\infty} = |\nabla F(y)|^{-1} dy.$$ From now on, we will be calling $\mu_{\infty}$ \textit{the Bateman Horn measure}.
\end{remark}

\section{A prime number theorem on cones}

In this section, we prove a prime number theorem on particularly structured regions for the determinant polynomial on $\Mat_n(\R)$; in particular we look at \textit{cones}. The reason for that is that we can solve the Linnik problem for this space (as mentioned in \S2.1). We will explain in the end of the section that the case of the Pfaffian follows the same lines. \\

For a nice compact subset $\Omega\subset V_1(\R)=\SL_n(\R)$, we define $$\pi_{\Mat_n,\det}(T,\Omega)=\#\{A\in T\Omega\cap \Mat_n(\Z): \text{ $\det A$ is prime}\}.$$ 

\begin{thm}\label{PNT cones Hardy-Littlewood}
Let $\Omega$ be a nice compact subset of $\SL_n(\R).$ Then as $T\rightarrow\infty$, the Bateman-Horn conjecture holds on $T\Omega$: 
    $$\pi_{\Mat_n,\det}(T,\Omega)\sim \sing_{\Mat_n,\det} \cdot \int_{T\Omega} \frac{1}{\log^+(\det(\bx))}d\bx,$$
    where $$\sing_{\Mat_n,\det} = \prod_{p} \left(\frac{1-p^{-n^2}\#\{\bx \in \Mat_n(\F_p): \det(\bx)\equiv 0 \bmod p\}}{(1-1/p)}\right).$$
\end{thm}

Next, we recall the following corollary of Theorem \ref{Oh's Theorem}, as proved by Oh in Example 4.1 of \cite{Oh}. In particular, the following shows that $(\Mat_n, \det)$ is a Hardy-Littlewood system and gives us an asymptotic count for the integral points on the level sets.

\begin{thm}[Oh \cite{Oh}, Example 4.1]\label{thm: Oh det}
For any nice compact subset $\Omega \subset V_1(\R) = \SL_n(\R)$, as $m\rightarrow\infty$: 
$$\#V_m(\Z) \cap m^{1/n}\Omega = (1+o_\Omega(1)) \cdot \mu_\infty(\Omega) \cdot m^{n-1} \cdot \sing_m,$$
where $$\sing_m = \prod_{q \text{ prime}} \lim_{k\rightarrow\infty} \frac{\#\{X\in \Mat_n(\Z/q^k\Z): \det(X) = m\}}{q^{k(n^2-1)}}$$
and $\mu_\infty$ is the Bateman-Horn measure. 
\end{thm}

We note that the singular series in Theorem \ref{PNT cones Hardy-Littlewood} differs in definition from the singular series in Theorem \ref{thm: Oh det}. In the next section, we show that these products are not so different. 

\subsection{Stability of the singular series}\label{sec: stab sing series}
We remark that our desired singular series from the Bateman-Horn conjecture (Conjecture \ref{conj: Bateman-Horn}) satisfies:
$$\sing_{\Mat_n,\det} = \prod_{p} \frac{p^{-n^2+1}\#\{v\in \Mat_n(\F_p):\det(v)\in \F_p^\times\}}{\varphi(p)}.$$ On the other hand, we also have the following Hardy-Littlewood singular series for any $m\in\Z$: $$\sing_m = \prod_{q\text{ prime}} \lim_{k\rightarrow\infty} \frac{\#\{v\in \Mat_n(\Z/q^k\Z): \det(v)=m\}}{q^{k(n^2-1)}}.$$ We will now show that these two quantities are similar when we specify $m$ to be a prime $p$. 
\begin{lemma}\label{lem: stability of constants}
For any sufficiently large prime $p$, the singular series $\sing_{p}$ coming from the Hardy-Littlewood property of $(\Mat_n,\det)$ satisfies $$\sing_p = \sing_{\Mat_n,\det}\cdot (1+O(p^{-\frac{1}{2}})).$$
\end{lemma}
\begin{proof}
We first notice the following local equidistribution property of $\Mat_n$: For any prime power $q^k$ and any $0\leq a,b<q^k$ with $(a,q^k)=(b,q^k)$, we have that \begin{align}\#\{v\in \Mat_n(\Z/q^k\Z): \det(v)=a\}&=\#\{v\in \Mat_n(\Z/q^k\Z): \det(v)=b\}.\label{Local Equidistribution}\end{align} It is then easy to see that for $q$ a prime that is not $p$: $$\#\{v\in \Mat_n(\F_q): \det(v)=p\} = \frac{\#\{v\in \Mat_n(\F_q): \det(v)\in \F_q^\times\}}{\varphi(q)},$$ and  $$\lim_{k\rightarrow\infty}\frac{\#\{v\in \Mat_n(\Z/q^k\Z): \det(v)=p\}}{q^{k(n^2-1)}} = \frac{\#\{v\in \Mat_n(\F_q): \det(v)\in \F_q^\times\}}{\varphi(q)}.$$
We recognize the right hand side as the $q$-factor in the Euler product of $\sing_{\Mat_n,\det}.$ We now prove Eq. \ref{Local Equidistribution}: 
We do this for $a=k=1,$ since the general case is similar. We consider the map
$$\{A\in \Mat_n(\Z/q\Z):\det(A)=1\}\rightarrow \{B\in \Mat_n(\Z/q\Z):\det(B)=b\}$$
given by $$A \mapsto \begin{pmatrix}
    b&0&... &0\\0&1&...&0\\\vdots&\vdots&\vdots&\vdots\\ 0&0&...&1
\end{pmatrix}A$$ for $A$ with $\det(A) =1$. Since $\gcd(b,q)=1$, this matrix $\text{diag}(b,1,...,1)$ is invertible in $\Mat_n(\F_q)$, so this gives us a bijection between the two sets. Thus, we get the equidistribution property. It follows that the local density at $q$ of the singular series $\sing_p$ and $\sing_{\Mat_n,\det}$ match at all prime $q\neq p$.

So, it remains to check that the factors almost match at the prime $p$. Since $\{\det(v) = 0\}$ is an absolutely irreducible variety over $\overline{\Q}$, we know that for $p$ sufficiently large this will be an absolutely irreducible variety over $\F_p$. Hence, the Lang-Weil bound \cite{LangWeil} implies that the following holds: \begin{equation}\label{eq: lang-weil} \#\{v\in \Mat_n(\Z/p\Z): \det(v)=p\}=p^{\dim(V)-1}(1+O(p^{-1/2})).\end{equation} Thus, we know that $$\frac{p^{-\dim(V)+1}\#\{v\in V(\F_p):F(v)\in \F_p^\times\}}{\varphi(p)}=1+O(p^{-\frac{3}{2}}).$$ 

On the other hand, from the equidistribution property (\ref{Local Equidistribution}), for $k\geq 2$ we have that
 $$\lim_{k\rightarrow\infty}\frac{\#\{v\in \Mat_n(\Z/p^k\Z): \det(v)\equiv p \bmod p^k\}}{p^{k(n^2-1)}} =\frac{\#\{v\in \Mat_n(\Z/p^2\Z): \det(v)\equiv p \bmod p^2\}}{p^{2(n^2-1)}}.$$
 Evaluating the right-hand side, we have that the expression equals:
 \begin{align*} 
 &\frac{\#\Mat_n(\Z/p^2\Z)}{\varphi(p)p^{2(n^2-1)}}-\frac{\#\{v\in \Mat_n(\Z/p^2\Z): \det(v)\in (\Z/p^2\Z)^{\times}\}}{\varphi(p)p^{2(n^2-1)}} -\frac{\#\{v\in \Mat_n(\Z/p^2\Z): \det(v)\equiv 0 \bmod p^2\}}{\varphi(p)p^{2(n^2-1)}} \\ &= p+1+O(p^{-1}) -\frac{p\#\{v\in \Mat_n(\F_p): \det(v)\in \F_p^\times\}}{\varphi(p)p^{n^2-1}}-\frac{\#\{v\in \Mat_n(\Z/p^2\Z): \det(v)\equiv 0 \bmod p^2\}}{\varphi(p) p^{2(n^2-1)}}\\&= 1-\frac{\#\{v\in \Mat_n(\Z/p^2\Z): \det(v)\equiv 0 \bmod p^2\}}{p^{2n^2-1}}+O(p^{-\frac{1}{2}}).
 \end{align*}
In the second equality, we have used (\ref{Local Equidistribution}) to handle the middle term. In the third equality, we have applied (\ref{eq: lang-weil}). It now suffices to show that for $p$ sufficiently large, 
\begin{equation}\label{eq: upp bd zero set}\#\{v\in \Mat_n(\Z/p^2\Z): \det(v)\equiv 0 \bmod p^2\}=O(p^{2n^2-2}),\end{equation}
since it will follow that 
$$\lim_{k\rightarrow\infty}\frac{\#\{v\in \Mat_n(\Z/p^k\Z): \det(v)\equiv p \bmod p^k\}}{p^{k(n^2-1)}} = 1+ O(p^{-1/2}).$$

Let $H=\{v\in \Mat_n(\F_p):\det(v) \equiv 0 \bmod p\text{ and $\nabla{\det}(v)\neq \b0 \bmod p$}\}$ denote the points of $\{v\in\Mat_n(\F_p): \det(v)\equiv 0 \bmod p\}$ that admit a Hensel lifting (see \cite{conrad}), and let $H^c=\{v\in \Mat_n(\F_p):\det(v) \equiv 0 \bmod p\text{ and $\nabla{\det}(v)= \b0 \bmod p$}\}$ denote its complement. We then have $$\#\{v\in \Mat_n(\Z/p^2\Z): \det(v)\equiv 0 \bmod p^2\}\leq \#H\cdot p^{n^2-1}+\#H^c\cdot p^{n^2}.$$
By the trivial bound we obtain $$\#H\leq \#\{v\in \Mat_n(\F_p):\det(v) \equiv 0 \bmod p\}\ll p^{n^2-1}.$$ Therefore, in order to show (\ref{eq: upp bd zero set}), it suffices to show $$\#H^c=O(p^{n^2-2}).$$

Indeed, since $\{v\in \Mat_n(\C): \det(v) =0\}$ is geometrically irreducible over $\C$, we know that $\{v\in \Mat_n(\F_p):\det(v) \equiv 0 \bmod p\}$ is geometrically irreducible over $\overline{\F_p}$ for all sufficiently large primes $p$. So, the codimension of the singular locus of $\{v\in \Mat_n(\F_p): \det(v)=0\bmod p\}$ will be at least one for such primes. As such, we know that $\dim(\{v\in \Mat_n(\F_p): \det(v)\equiv 0 \bmod p, \nabla \det(v) = \b0\bmod p\}) \leq n^2-2.$ Hence, we have that $$\#H^c=\left|\{v\in \Mat_n(\F_p): \det(v)\equiv 0 \bmod p, \textrm{ and }\nabla \det(v) = \b0\bmod p\}\right\vert\ll p^{n^2-2}.$$ This completes the proof of (\ref{eq: upp bd zero set}).

Finally, we put all the local densities together. For $p\neq q$, we have that $$\lim_{k\rightarrow\infty}\frac{\#\{v\in \Mat_n(\Z/q^k\Z): \det(v)=p\}}{q^{k(n^2-1)}} = \frac{\#\{v\in \Mat_n(\F_q): \det(v)\in \F_q^\times\}}{\varphi(q)}.$$
On the other hand, the local density at $p$ is given by $$\lim_{k\rightarrow\infty}\frac{\#\{v\in \Mat_n(\Z/p^k\Z): \det(v)\equiv p \bmod p^k\}}{p^{k(n^2-1)}} = 1+ O(p^{-1/2}).$$
Thus, we have that $$\sing_p = \sing_{\Mat_n,\det}(1+O(p^{-1/2})).$$
\end{proof}

\subsection{Counting primes in cones}\label{sec: PNT cones} Using the results of this section we are now in a position to prove Theorem \ref{PNT cones Hardy-Littlewood}: 

\begin{proof}[Proof of Theorem \ref{PNT cones Hardy-Littlewood}]
From Theorem \ref{thm: Oh det}, for $T$ sufficiently large, we know that
\begin{align*}\pi_{\Mat_n,\det}(T\Omega) &= (1+o_\Omega(1))\mu_\infty(\Omega)\sum_{p\leq T^{n}}p^{n-1} \cdot \sing_p\\
&= (1+o_{\Omega}(1)) \mu_\infty(\Omega) \sing_{\Mat_n,\det} \sum_{p\leq T^{n}} p^{n-1}(1+O(p^{-1/2}))\\
&= (1+o_\Omega(1))\mu_\infty(\Omega)\sing_{\Mat_n,\det}\cdot \frac{T^{n}}{n^2\log(T)} \\
&= (1+o_\Omega(1)) \sing_{\Mat_n,\det} \int_{T\Omega} \frac{dx}{\log^+(\det(x))}.\end{align*}
The second equality follows from Lemma \ref{lem: stability of constants} and the fourth equality comes from Lemma \ref{lem: the real part for cones PNT}. We note that $\sing_{\Mat_n,\det}$ is the singular series given by the Bateman-Horn conjecture.  
\end{proof}
Finally, we remark that since $$\#\{v\in \Mat_n(\F_p):\det(v)\in \F_p^\times\} = \prod_{j=0}^{n-1}(p^n-p^j),$$
we can evaluate the singular series as $$\sing_{\Mat_n,\det} = \prod_{j=2}^n \zeta(j)^{-1}.$$

\subsection{Modifications for the Pfaffian polynomial}
The analysis for the Pfaffian on $2n\times 2n$ skew-symmetric matrices follows along the same lines as the proof of Theorem \ref{PNT cones Hardy-Littlewood} described above. In this subsection, we prove the corresponding prime number theorem on cones. 

Let $V = \Skew_{2n} = \{X\in \Mat_{2n}: X^T = -X\}$. 
For a nice compact subset $\Omega \subset V_1(\R)$, we define $$\pi_{\Skew_{2n},\Pff}(T,\Omega) = \#\{A\in T\Omega \cap \Skew_{2n}(\Z): \Pff(A) \text{ is prime}\}.$$
\begin{thm}\label{thm: PNT Pfaffian cones}
Let $\Omega$ be a nice compact subset of $V_1(\R) = \{X\in \SL_{2n}(\R):X^T = -X\}.$ Then as $T\rightarrow\infty$, the Bateman-Horn conjecture holds on $T\Omega$: 
$$\pi_{\Skew_{2n},\Pff}(T,\Omega) \sim \sing_{\Skew_{2n},\Pff} \cdot \int_{T\Omega} \frac{1}{\log^+(\Pff(\bx))}d\bx,$$
where $$\sing_{\Skew_{2n},\Pff} = \prod_{p} \left(\frac{1-p^{-(2n^2-n)}\#\{\bx\in \Skew_{2n}(\F_p):\Pff(x)\equiv 0 \bmod p\}}{(1-1/p)}\right).$$
\end{thm}

 We start by noting that Theorem \ref{Oh's Theorem} also implies that this system is Hardy-Littlewood.
In particular, we have that by Example 4.2 of Oh \cite{Oh}:
\begin{thm}[Oh \cite{Oh}, Example 4.2]\label{thm: Oh Pff}
  For any nice compact subset $\Omega\subset V_1(\R)$, as $m\rightarrow\infty$: 
    $$\#V_m(\Z) \cap m^{1/n}\Omega = (1+o_\Omega(1)) \cdot \mu_\infty(\Omega) \cdot m^{2n-2} \cdot \sing_m,$$
    where $$\sing_m = \prod_{q\text{ prime}} \lim_{k\rightarrow\infty} \frac{\#\{X\in \Skew_{2n}(\Z/q^k\Z):\Pff(X) = m\}}{q^{k(n^2-1)}}$$
    and $\mu_\infty$ is the Bateman-Horn measure. 
\end{thm}

Next, we want to prove that these singular series $\sing_{\Skew_{2n},\Pff}$ and $\sing_m$ are similar when $m=p$ is a large prime. 

\begin{lemma}\label{lem: stability Pff}
    For any sufficiently large prime $p$, we have that $$\sing_p = \sing_{\Skew_{2n},\Pff}\cdot (1+O(p^{-1/2})).$$
\end{lemma}
 \begin{proof}
 We note that, in the proof of Lemma \ref{lem: stability of constants}, the only aspect that was specific to $(\Mat_n,\det)$ was the local equidistribution property (\ref{Local Equidistribution}). So, it suffices to prove a local equidistribution property for $(\Skew_{2n},\Pff):$ For any $0\leq a,b <q^k$ and $\gcd(a,q^k) = \gcd(b,q^k)$, we have that $$\#\{X\in \Skew_{2n}(\Z/q^k\Z): \Pff(X) = a\} = \#\{X\in \Skew_{2n}(\Z/q^k\Z): \Pff(X)= b\}.$$

 Again, we show this fact when $q$ is a prime and $(b,q)=1$:
\begin{equation}\label{eq: local solution Pff equidistribution}
    \#\{X \in \Skew_{2n}(\F_q): \Pff(X) \equiv 1 \bmod q\} = \#\{X\in \Skew_{2n}(\F_q): \Pff(X) \equiv b \bmod q\}.
\end{equation}
The cases when $a\neq 1$ and $k\neq 1$ can be handled by the same proof. 

We consider the map
$$\varphi:\{A\in \Skew_{2n}(\F_q):\Pff(A) \equiv 1 \bmod q\} \rightarrow \{B\in \Skew_{2n}(\F_q): \Pff(B) \equiv b\bmod q\}$$
given by:
$$\varphi(A) = \begin{pmatrix}
    b&0&... &0\\0&1&...&0\\\vdots&\vdots&\vdots&\vdots\\ 0&0&...&1
\end{pmatrix} A \begin{pmatrix}
    b&0&... &0\\0&1&...&0\\\vdots&\vdots&\vdots&\vdots\\ 0&0&...&1
\end{pmatrix}.$$
Clearly, if $A^T = -A$, then $\varphi(A)^T = -\varphi(A)$. We also know that $$\Pff(A) = \sum_{j=2}^{2n} (-1)^j a_{1j} \Pff(A_{\hat{1}\hat{j}}),$$
where $A_{\hat{i}\hat{j}}$ is the minor created by removing the $i$th and $j$th row and column. Since this map $\varphi$ results in multiplying the first row and column both by $b$, we can see that $$\Pff(\varphi(A)) = \sum_{j=2}^{2n} ba_{1j} \Pff(\varphi(A)_{\hat{i}\hat{j}}) = b\sum_{j=2}^{2n} a_{1j} \Pff(A_{\hat{i}\hat{j}}) = b\Pff(A).$$
So, we do indeed have $\varphi: \{X\in \Skew_{2n}(\F_q): \Pff(X) \equiv 1 \bmod q\} \rightarrow \{X\in \Skew_{2n}(\F_q): \Pff(X) \equiv b \bmod q\}.$ Since $\text{diag}(b,1,1,...,1)$ is invertible, $\varphi$ is a bijection between the two sets, so we get that if $(q,b)=1$, the local equidistribution property (\ref{eq: local solution Pff equidistribution}) holds. 
\end{proof}

Now, we have all the ingredients necessary to prove Theorem \ref{thm: PNT Pfaffian cones}. 
\begin{proof}[Proof of Theorem \ref{thm: PNT Pfaffian cones}]
    From Theorem \ref{thm: Oh Pff}, for $T$ sufficiently large, we know that
\begin{align*}\pi_{\Skew_{2n},\Pff}(T\Omega) &= (1+o_\Omega(1))\mu_\infty(\Omega)\sum_{p\leq T^{n}}p^{2n-2} \cdot \sing_p\\
&= (1+o_{\Omega}(1)) \mu_\infty(\Omega) \sing_{\Skew_{2n},\Pff} \sum_{p\leq T^{n}} p^{2n-2}(1+O(p^{-1/2}))\\
&= (1+o_\Omega(1))\mu_\infty(\Omega)\sing_{\Skew_{2n},\Pff}\cdot \frac{T^{n(2n-1)}}{n(2n-1)\log(T)} \\
&= (1+o_\Omega(1)) \sing_{\Skew_{2n},\Pff} \int_{T\Omega} \frac{d\bx}{\log^+(\Pff(\bx))}.\end{align*}
The second equality follows from Lemma \ref{lem: stability Pff} and the fourth equality comes from Lemma \ref{lem: the real part for cones PNT}. We note that $\sing_{\Skew_{2n},\Pff}$ is the singular series given by the Bateman-Horn conjecture.  
\end{proof}

Finally, we evaluate the singular series: $$\sing_{\Skew_{2n},\Pff} = \prod_{p} \frac{\#\{X\in \Skew_{2n}(\F_p): \det(X) \in \F_p^\times\}}{p^{n(2n-1)}(1-1/p)}.$$ Applying the result of MacWilliams in \cite{MacWilliams}, 
    \begin{equation}\label{eq: pfaffian local count}\#\{X\in \Skew_{2n}(\F_p): \det(X) \in \F_p^\times\} = \prod_{i=0}^{n-1}(p^{2n-1}-p^{2i}).\end{equation}
    Thus, we get that $$\sing_{\Skew_{2n},\Pff} = \prod_{p} (1-1/p)^{-1}\prod_{i=0}^{n-1}(1-p^{1+2i-2n}) = \prod_{\substack{3\leq j\leq 2n-1\\ j\textrm{ odd}}}\zeta(j)^{-1}.$$

\section{$\epsilon$-cuttings}\label{sec: eps cutting}

In this section, we will no longer require the dynamical set up of \S\ref{sec: general framework}. We assume that $V$ is a linear vector space of dimension larger than one and $F$ is a homogeneous polynomial that is irreducible over $\overline{\Q}$.  
Again, we denote our level sets:
$$V_m = \{v\in V: F(v)=m\}.$$
We note that since $F$ is homogeneous, we know that $$V_m = m^{1/\deg(F)} V_1.$$

Our goal for the following two sections will be to prove the following theorem:
\begin{thm}\label{thm: cones to boxes}
     Assume that for any nice compact $\Omega\subset V_1(\R)$, as $T\rightarrow \infty$, we have that $$\pi_{V,F}(T\Omega) = (1+o_\Omega(1)) \sing_{V,F} \cdot \int_{T\Omega} \frac{1}{\log^+(F(x))}dx.$$
If a compact set $\calR_0 \subset V(\R)$ has the following properties: 
\begin{enumerate}
    \item If $x\in\calR_0$ and $t\in[0,1]$, then $tx\in\calR_0,$
    \item The compact set $[0,\infty)\calR_0\cap V_1$ is nice,
\end{enumerate}
    then we have that as $T\rightarrow\infty$, 
    $$\pi_{V,F}(T\calR_0) = (1+o(1)) \sing_{V,F} \cdot \int_{T\calR_0} \frac{1}{\log^+(F(x))}dx.$$

    Moreover, if we also have that for any nice compact $\Omega \subset V_{-1}(\R)$, as $T\rightarrow\infty$, the following holds: $$\#\{x \in T\Omega: -F(v) \text{ is prime}\} =(1+o_\Omega(1)) \sing_{V,F} \int_{T\Omega} \frac{1}{\log^+(|F(\bx)|)}dx,$$
    then Conjecture \ref{conj: Bateman-Horn} holds for the compact regions $\calR_0\subset V(\R)$.
\end{thm}
\begin{remark}
    We note that the Euclidean box $\calB = [0,1]^{n}\subset \R^n$ would satisfy the conditions $(1),(2)$ above; hence, a consequence of Theorem \ref{thm: cones to boxes} is that under the assumptions given, Conjecture \ref{conj: DS Bateman-Horn} holds. 
\end{remark}

First, we will approximate the region $T\calR_0$ with such cones $T\Omega$ using $\epsilon$-cuttings.
\begin{defn}
    For $\Omega\subset V_1(\R)$, a nice compact subset, we define the \textbf{height} of the subset as $$\height(\Omega) = \sup_{A\in \Omega} \|A\|,$$
    where $\|.\|$ is the $L^\infty$ norm on $V(\R)$.
\end{defn}

\begin{defn}\label{def: epsilon-cutting}
An $\epsilon$-\textbf{cutting} of $\calR_0$ is a finite set of disjoint nice connected compact subsets of $V_1(\R)$, denoted as $$\calC_\epsilon = \{\Omega_{i}\}_{i=1}^{N(\epsilon)},$$ such that
$$\calR_0 = \calE \bigcup_{\Omega\in \calC_\epsilon} [0,1/\height(\Omega)]\Omega,$$
where the exceptional set satisfies that $|\calE| \leq \epsilon.$
\end{defn}
\begin{remark}
    We note that such an $\epsilon$-cutting of $\calR_0$ allows us to cut up $T\calR_0$ for any $T\geq 1$; we have that $$T\calR_0 = T\calE\cup_{\Omega\in C_\epsilon} [0,T/\height(\Omega)]\Omega,$$
    and that $|T\calE|\leq \epsilon T^{\dim(V)}.$
\end{remark}
From now on we will be assuming that the compact set $\calR_0$ has the following properties (as in Theorem \ref{thm: cones to boxes}): 
\begin{enumerate}
    \item If $x\in\calR_0$ and $t\in[0,1]$, then $tx\in\calR_0.$
    \item The compact set $[0,\infty)\calR_0\cap V_1$ is nice.
\end{enumerate}

Next, we show the existence of an $\epsilon$-cutting. 

\begin{prop}\label{prop: existence of epsilon cutting}
Fix an $\epsilon>0$ sufficiently small. Let $\calR_0\subset V(\R)$ satisfy assumptions (1) and (2) above. Then, there exists an $\epsilon$-cutting of $\calR_0\subset V(\R)$.
\end{prop}
\begin{proof}
Let $L = \calR_0 \cap \{v\in V(\mathbb{R}): F(v)=0\}.$ Since $\calR_0$ is compact, $L$ is a compact set. Moreover, $\vol(L)=0$, since $L$ is contained in a subvariety of codimension one. Hence, we can find a sufficiently small $\delta$, such that the set $$L_{\delta}=\{v\in \calR_0: \textrm{dist}(v,L)<\delta \}$$ has volume less than $\epsilon/2$. We set $\calE_0 = L_\delta.$

Let us consider the set $S = \calR_0 - L_\delta$ and split $S$ into its connected components and such that each component is contained in a separate orbit of the action of $G(\R)$ on $V(\R)$, which will produce a finite number $k_{V}$ of $S = \cup_{i\leq k_{V}} S_i$, where $k_{V}$ is independent of $\epsilon$. We will abuse notation and denote one of these components as $S$. Now, we wish to split $S$ up further. 

We consider $$S_1 = \R_+ S \cap V_1(\R)$$ and claim that $S_1$ is a compact set. Indeed, since $S$ is a compact set disjoint from the zero locus of $F$, we know that $$m_0(\epsilon)=\min_{u\in S}\left\vert F(u)\right\vert>0.$$ Thus, $S_1 \subset V_1(\R) \cap [0,1/m_0(\epsilon)^{1/\deg(F)}]S$ and hence it is compact. So, we have that $$0<m(\epsilon)=\inf_{v\in S_1} \|v\|\leq\sup_{v\in S_1} \|v\| = M(\epsilon)<\infty.$$ 

Let $\epsilon'$ be a function of $\epsilon$ that we will determine later. Let us split the interval $[m(\epsilon), M(\epsilon)]$ into intervals of length at most $\eta(\epsilon) = \epsilon'\min(1,M(\epsilon)^{-1})$ and enumerate these intervals as $I_k=[m_k-\eta(\epsilon),m_k]$ (We can assume that $M(\epsilon)-m(\epsilon)$ is an integer multiple of $\eta(\epsilon)$, after passing to a possibly smaller constant). Then we set $$S_{1,k} = \{v\in S_1: \|v\| \in I_k\}.$$ and finally $$S'=\bigcup_{k} [0, 1/m_k] S_{1,k}.$$
Note that $S'\subset S\subset \cup_k [0,1/(m_k-\eta(\epsilon))] S_{1,k}$. Observe now that $$\vol([1/m_k, 1/(m_k-\eta(\epsilon))]S_{1,k})\leq (m_k-\eta(\epsilon))^{-\dim(V)} - m_k^{-\dim(V)} \leq c_V \eta(\epsilon)m_k^{-\dim(V)},$$
where $c_V$ is a constant in terms of $\dim(V)$. So, we know that$$\sum_{k} \eta(\epsilon)m_k^{-\dim(V)} \leq C \eta(\epsilon) \int_{m(\epsilon)}^{M(\epsilon)} m^{-\dim(V)} dm\leq C_V' \eta(\epsilon),$$
where $C, C_V'>0$ are constants. 
Thus, we have that by squeezing, $$\vol(S-S') \leq c_VC_V' \epsilon'.$$
We choose $\epsilon' \leq \frac{\epsilon}{2c_VC_V'}$ and set $\calE = (S-S') \cup \calE_0.$ For such a choice, we have that $$\calR_0 = \calE \cup \bigcup_{i}\bigcup_{k} [1,\height(S_{i,k})]S_{i,k},$$
with $\vol(\calE)\leq \epsilon.$
\end{proof}

\section{An upper bound sieve}\label{sec: sieve}
In this section, we will first establish a general upper bound for the number of prime values of an irreducible multivariate polynomial; this upper bound will be of the same size as what the Bateman-Horn conjecture (Conjecture \ref{conj: DS Bateman-Horn}) predicts. We note that we do not need the set up of the general framework for this bound.

Let $F(\bx)\in \Z[\bx]$ be an irreducible polynomial in $n\geq 2$ variables. Let us denote the box of height $T$ by $$B(T) := \{\bx\in \R^n: \|\bx\|_{\infty}\leq T\}.$$
Let us also write $\scrP$ as the set of primes.
Then we get the following upper bound for the number of prime values of $F(\bx)$:
\begin{prop}\label{prop: upper bound}
    Let $\calR\subset B(T)$ be a convex region. Then there is a constant $c_F>0$ such that $$\#\{\bx\in \calR(\Z): F(\bx)\in \scrP\} \leq c_F\left(\frac{|\calR|}{\log(T)} + T^{n-1/2}\right).$$ 
\end{prop}
\begin{proof}
    Let $z= T^{\frac{1}{10n}}$ and we write $$P(z) = \prod_{\substack{p\in \scrP\\ p\leq z}} p.$$ We can bound the number of prime values of $F$ in $\calR$ using the $z$-rough numbers as follows: $$\#\{\bx\in \calR(\Z):F(\bx)\in \scrP\} \leq \#\{\bx\in \calR(\Z): F(\bx) \in \scrP\cap [1,z]\} + \#\{\bx\in \calR(\Z): \gcd(F(\bx),P(z))=1\},$$

First we apply Lemma 6.8 of Opera de Cribro \cite{OperaDeCribro} to the count of $z$-rough numbers. We then know that for a fixed constant $C>0$, the following inequality holds: 
\begin{multline}\label{eq: sieve application}
    \#\{\bx\in \calR(\Z):\gcd(F(\bx),P(z))=1\}\leq C |\calR| \prod_{p\leq z} \left(1-\frac{\#\{\bx\in \F_p^n: F(\bx)\equiv 0 \bmod p\}}{p^{n}}\right) \\ + O\left(\sum_{d\leq z}\mu^2(d) |r_d(\calR)|\right),
\end{multline}
where the $r_d(\calR)$ is the remainder term: $$r_d(\calR) = \#\{\bx\in \calR(\Z): F(\bx)\equiv 0 \bmod d\} - |\calR|\cdot \frac{\#\{\bx\in (\Z/d\Z)^n: F(\bx)\equiv 0 \bmod d\}}{d^n}.$$ 

First, we consider the product over primes $p\leq z$: 
\begin{align*}
    \prod_{p\leq z} \left(1-\frac{\#\{\bx\in \F_p^n: F(\bx)\equiv 0 \bmod p\}}{p^{n}}\right) &\ll_F \exp(-\sum_{p\leq z} \frac{\#\{\bx\in \F_p^n:F(\bx)\equiv 0 \bmod p\}}{p^n} + O(p^{-2}))\\
    &\ll_F \exp(-\sum_{p\leq z} p^{-1} + O(p^{-3/2})) \\
    &\ll_F \log(z)^{-1}.
\end{align*}
In the first inequality, we use that $$\varrho(p) := \#\{\bx\in \F_p^n: F(\bx)\equiv 0 \bmod p\} \leq dp^{n-1},$$
so we know that $$\log(1-\frac{\varrho(p)}{p^n}) = -\frac{\varrho(p)}{p^n} + O(p^{-2}).$$
For the second inequality, we use that since $F$ is absolutely irreducible, we have an even better estimate for $\#\{\bx\in \F_p^n: F(\bx)\equiv 0 \bmod p\}$ for $p$ sufficiently large in terms of $F$:
$$\#\{\bx\in \F_p^n: F(\bx)\equiv 0 \bmod p\} = p^{n-1} + O(p^{n-3/2}).$$
Moreover, for any $\alpha>1$, we know that $$\sum_{p\leq z} p^{-\alpha} = O(1),$$
whereas for $\alpha=1$, we have $$\sum_{p\leq z} p^{-1} = \log\log(z) + O(1).$$
Together these give the final inequality. 

Next, we need a nontrivial level of distribution result: $$\sum_{d\leq T^{1/10n}}\mu^2(d) |r_d(\calR)| \ll_F T^{n-1/2}.$$
For a fixed $d$, we can write 
    \begin{equation}\label{eq: remainder}r_d(\calR) = \sum_{\substack{\alpha\bmod d\\ F(\alpha)\equiv 0 \bmod d}} \left(\sum_{\substack{\bx\in \calR(\Z) \\ v\equiv \alpha\bmod d}}1 - \frac{|\calR|}{d^{n}}\right).\end{equation}

We estimate the internal sum using geometry of numbers. We remark that if we define $$\Lambda(\alpha,d) := \{\bx\in \Z^n: \bx \equiv \alpha\bmod d\},$$
then $\Lambda(\alpha,d)$ is a lattice of determinant $d^{n}.$ Now, since $\calR\subset B(T)$, we know that the longest ``sidelength'' of the region is $\ll T$. Moreover, it is a convex region and thus, we get that $$\#\{\bx\in \calR(\Z)\cap \Lambda(\alpha,d)\} = \frac{|\calR|}{d^{n}} + O\left(\max_{1\leq k\leq n} \frac{T^{n-k}}{\lambda_1\hdots \lambda_{n-k}}\right),$$
where $\lambda_1,...,\lambda_{n}$ are the successive minima of the lattice $\Lambda(\alpha,d).$ Since $\lambda_i\geq 1$ for each $i$, we use the trivial bound that $$\#\{v\in \calR(\Z)\cap \Lambda(\alpha,d)\} - \frac{|\calR|}{d^{n}} \ll T^{n-1}.$$

Thus, we have that (\ref{eq: remainder}) is bounded by $$|r_d(\calR)| \leq  \sum_{\substack{\alpha\bmod d\\ F(\alpha)\equiv 0 \bmod d}} T^{n-1}\ll T^{n-1} d^{n-1}.$$
Here we have used that $F(v)\equiv 0 \bmod d$ cuts out a subvariety of codimension $1$. So, we can calculate that $$\sum_{d\leq T^{1/10n}}\mu^2(d) |r_d(\calR)| \ll_F T^{n-9/10}\leq T^{n-1/2}.$$
Plugging the above bound into (\ref{eq: sieve application}), we achieve:
\begin{equation}\label{eq: result of sieve}\#\{\bx\in \calR(\Z): \gcd(F(\bx),P(z))=1\} \ll_F \frac{|\calR|}{\log(T)} + T^{n-1/2}.\end{equation}

Finally, we need to bound the contribution from the small primes less than $z$: 
\begin{align*}\#\{\bx\in \calR(\Z): F(\bx)\in \scrP\cap [1,T^{1/10n}]\} &\leq \sum_{m=1}^{T^{1/10n}} \#\{\bx\in \calR(\Z): F(\bx) = m\} \\ &\leq \sum_{m=1}^{T^{1/10n}} \#\{\bx\in B(T)(\Z): F(\bx)=m\}.\end{align*}
Let us write $\bx=(x_1,...,x_n)$ where we assume that $F(\bx) = G(x_1,...,x_{n-1},x_n)x_n + H(x_1,...,x_{n-1})$ for $G(x_1,...,x_{n-1},x_n)$ nontrivial. Then we can write the sum above as:
$$\leq \sum_{|x_1|\leq T}\hdots \sum_{|x_{n-1}|\leq T} \sum_{m=1}^{T^{1/10n}} \#\{|x_n|\leq T: F(x_1,...,x_{n-1},x_n) = m\}.$$
Now if $F(x_1,...,x_{n-1},x_n)$ as a polynomial in $x_n$ has degree $d \geq 1$, we have that $$\#\{|x_n|\leq T: F(x_1,...,x_n)=m\}\leq d$$
for any value of $d$. So, these values of $x_1,...,x_{n-1}$ will contribute to an upper bound of $T^{n-1+1/10n}$ values.

On the other hand, if $F(x_1,...,x_{n-1},x_n)=c$ is a constant polynomial in $x_n$, for $m=c$, we get that $$\#\{|x_n|\leq T: F(x_1,...,x_{n-1},x_n)=c\}\leq T$$
and otherwise $$\#\{|x_n|\leq T: F(x_1,...,x_{n-1},x_n)=m\}=0.$$
However, for $F(x_1,...,x_{n-1},x_n)$ to be a constant polynomial, $x_1,...,x_{n-1}$ must satisfy $\tilde{G}(x_1,...,x_{n-1})=0$ where $\tilde{G}(x_1,...,x_{n-1})$ is the coefficient of any positive power of $x_n$ in $F(x_1,...,x_{n-1},x_n)$. We know that $$\#\{|x_i|\leq T: \tilde{G}(x_1,...,x_{n-1}) = 0\} \ll_{\tilde{G}} T^{n-2},$$
since $\tilde{G}$ cuts out a codimension one variety. So, combining this with the above bound, we get that $$\#\{\bx\in \calR(\Z): F(\bx)\in \scrP\cap [1,T^{1/10n}]\} \leq \sum_{m=1}^{T^{1/10n}} \#\{\bx\in B(T)(\Z):F(\bx)=m\} \ll_{F} T^{n-1+1/10n}.$$
Combining the above bound with (\ref{eq: result of sieve}), we achieve final upper bound.

\end{proof}

We remark that for $F$ a homogeneous polynomial, this level of distribution result can be improved by studying the level of distribution of lattices, as done by Daniel in \cite{Daniel}, or by using bounds on exponential sums, as done by Friedlander and Iwaniec in Chapter 22 of \cite{OperaDeCribro}.

\subsection{Proof of Theorem \ref{thm: cones to boxes}}
From Proposition \ref{prop: existence of epsilon cutting}, we know that we have an $\epsilon$-cutting of $\calR_0$. Thus if we define $$\pi_{V,F}(T\calR_0):= \#\{v\in V(\Z)\cap T\calR_0: F(v) \textrm{ is prime}\},$$
then we have that for every $T>1$: $$\pi_{V,F}(T\calR_0) = \sum_{\Omega\in \calC_\epsilon} \pi_{V,F}(T/\height(\Omega)\Omega)) + \#\{v\in V(\Z)\cap T
\calE: F(v) \textrm{ is prime}\}.$$

First we handle the error term with a sieve argument. We note that in the construction of the $\epsilon-$cutting from Proposition \ref{prop: existence of epsilon cutting}, $\calE$ will be the finite union of convex regions contained in $\calR_0$.  Then applying Proposition \ref{prop: upper bound} on each of the convex regions, we get that there is a constant $C_{F,V}\geq 0$ such that $$\#\{v\in V(\Z)\cap T\calE: F(v)\text{ is prime}\} \leq C_{F,V} \left(\frac{|T\calE|}{\log(T)} + T^{\dim(V)-1/2} \right) \leq C_{F,V}\frac{ \epsilon T^{\dim(V)}}{\log(T)}.$$

Now, we handle our main term. Under our assumptions, we have that for $\Omega\subset V_1(\R)$ a nice compact subset, $$\pi_{V,F}(T/\height(\Omega)\Omega) = (1+o_\Omega(1))\sing_{V,F} \int_{T/\height(\Omega)\Omega} \frac{1}{\log^+(F(x))}dx.$$
Thus, our main term becomes $$\sum_{\Omega\in \calC_\epsilon} \pi_{V,F}(T/\height(\Omega)\Omega) = (1+o_\epsilon(1))\sing_{V,F} \int_{\cup_{\calC_\epsilon} T/\height(\Omega)\Omega} \frac{1}{\log^+(F(x))}dx.$$ 
By taking $\epsilon\rightarrow 0$, we obtain $$\pi_{V,F}(T\calR_0) = (1+o(1))\sing_{V,F} \int_{T\calR_0} \frac{1}{\log^+(F(x))}dx.$$
By the same argument, if the second assumption holds in Theorem \ref{thm: cones to boxes}, then by using an $\epsilon$-cutting of $\calR_0$ with nice compact subsets $\Omega\subset V_{-1}(\R)$, the sum over negative primes achieves the final statement of the theorem. \qed

\subsection{Proof of Theorem \ref{thm: PNT det Mat} and Theorem \ref{thm: PNT Pff Skew}}\label{sec: PNT mat,skew}
In this section, we combine Theorem \ref{PNT cones Hardy-Littlewood} and Theorem \ref{thm: PNT Pfaffian cones} with Theorem \ref{thm: cones to boxes} to achieve Theorem \ref{thm: PNT det Mat} and Theorem \ref{thm: PNT Pff Skew}. 

\begin{proof}[Proof of Theorem \ref{thm: PNT det Mat}]
    First, we note that Theorem \ref{PNT cones Hardy-Littlewood} gives us the first assumption of Theorem \ref{thm: cones to boxes}. Next, we remark that the box $\{\|X\|\leq T\}\subset \Mat_n(\R)$ satisfies the conditions on $\calR_0$ set by Theorem \ref{thm: cones to boxes}. Hence, we get the following result: 
    $$\pi_{\det}(T) = (1+o(1)) \cdot \prod_{j=2}^n \zeta(j)^{-1} \int_{\|X\|\leq T} \frac{1}{\log^+(\det(X))}dX.$$
\end{proof}

\begin{proof}[Proof of Theorem \ref{thm: PNT Pff Skew}]
    First, we note that Theorem \ref{thm: PNT Pfaffian cones} gives us the first assumption of Theorem \ref{thm: cones to boxes}. Next, we remark that the box $\{\|X\|\leq T\}\subset \Skew_{2n}(\R)$ satisfies the conditions on $\calR_0$ set by Theorem \ref{thm: cones to boxes}. Hence, we get the following result: 
    $$\pi_{\Pff}(T) = (1+o(1))\cdot \prod_{\substack{3\leq n\leq 2n-1\\ j\text{ odd}}} \zeta(j)^{-1} \cdot \int_{\substack{\|X\|\leq T\\ X^T = -X}} \frac{1}{\log^+(\Pff(X))}dX.$$
\end{proof}

\section{Background on symmetric matrices and the Siegel Mass Formula}\label{sec: Siegel}
Finally, we handle the more involved example of the determinant polynomial on symmetric matrices. Let us first discuss how this case does and does not fit into the general framework outlined in \S\ref{sec: general framework}. 

Define $G=SL_n(\R)$ and $V$ to be the space of symmetric matrices:
$$V = \{X\in \Mat_n: X^T=X\}.$$
We will take $F$ to be the determinant polynomial on $V$. Then we have that $g\in G= \SL_n$ acts on $V_1 = \{X\in \Mat_n: X^T=X, \det(X) = 1\}$ via $$X \mapsto g^T X g.$$
We can then see that $V_1(\R)$ is the finite union of $n+1$ orbits, where the orbits are parameterized by the signatures of the symmetric matrices. For $p+q=n$, let $\calO_{p,q}$ be the orbit of matrices with signature $(p,q).$
In this case, the conditions to apply Theorem 1.9 of Oh \cite{Oh} are not always satisfied -- in particular for the orbit of positive (or negative) definite symmetric matrices, the stabilizer of the generator $\text{diag}(1,1,...,1)$ is $\SO_n(\R)$ and thus compact. To resolve this issue, we use the Linnik equidistribution result of Einsiedler, Margulis, Mohammadi, and Venkatesh in \cite{EMMV} to get asymptotics for $$\{X\in\GL_n(\Z)\cap m^{\frac{1}{n}}\Omega:\text{ $X$ is symmetric positive definite}\},$$ for $\Omega$ some nice compact subset of $\SL_n(\R).$ To do so, we use the Siegel mass formula to evaluate the corresponding class numbers. Moreover, by evaluating the local factors of the Siegel mass formula we calculate the singular series. \\

Before proving the prime number theorem for the symmetric case, it is prudent to review some important facts about the Siegel mass formula. Recall that two integral positive definite symmetric matrices are in the same genus if they are equivalent over $\mathbb{Z}_p$ for all $p$ and over $\R$. There are finitely many equivalence classes in each genus. For a positive definite symmetric matrix $S$ of genus $\calG$, the mass $M(S)$ of $S$ is defined as $$M(S) := \sum_{S_k\in \calG} \frac{1}{\left\vert \textrm{Aut}(S_k)\right\vert},$$
where the sum is over representatives $S_k$ in the genus $\calG$ of $S$ and $\textrm{Aut}(S_k)$ is the automorphism group of $S_k$ in $\GL_n(\Z)$. We note that $M(S)$ only depends on the genus of $S$, and could be denoted as $M(\calG)$. Siegel's celebrated mass formula evaluates this quantity as a product of local densities. 
\begin{thm}[Siegel \cite{Siegel}]\label{thm: Siegel Mass formula}
  Let $n\geq 3$ and $S$ be a positive definite symmetric $n\times n$ matrix of determinant $D$. Then we have the following evaluation of $M(S):$
  $$M(S)= 2 \pi^{-n(n+1)/4} \prod_{k=1}^n \Gamma(k/2)  \prod_{p} \alpha_p(S)^{-1} \cdot D^{(n+1)/2},$$
  where $\alpha_p$ is the $p$-adic density at the prime $p$ defined as $$\alpha_p(S) := \frac{1}{2}\lim_{\ell \rightarrow\infty}p^{-\ell n(n-1)/2 }\#\{T\in \Mat_n(\Z/p^{\ell}\Z): T^T S T \equiv S \bmod p^{\ell}\}.$$
\end{thm}

For all but finitely many primes $p$ in the Siegel mass formula (for $p\nmid 2D$), the formula for $\alpha_p(S)$ has the following clean expression:
\begin{equation}\label{eq: standard p mass}
\alpha_p^{\std}(S)=\begin{cases}\prod_{k=1}^{(n-1)/2} (1-p^{-2k}), & n\text{ odd},\\ \prod_{i=1}^{n/2-1} (1-p^{-2k})\cdot (1-p^{-n/2}\legendre{(-1)^{n/2}D}{p}), & n\text{ even.} \end{cases}
\end{equation}
We will call this the \textbf{standard $p$-density}. Here the Legendre symbol $\legendre{(-1)^{n/2}D}{p}$ is interpreted as $0$ if $p\mid 2D.$ 

Evaluating $\alpha_p(S)$ at the nonstandard primes can be tricky; calculating the local factors of the Siegel mass formula for $p=2$ is a notorious problem. Fortunately, we will only need the case when the determinant $D$ is a prime and hence the only \textit{small} divisor we have to deal with is $2$. We note that both Kitaoka \cite{KY} and Kovaleva \cite{Kovaleva} evaluated these densities at odd primes; additionally, Kitaoka evaluates $\alpha_p(S)$ for all primes when $n$ is odd.

\subsection{The density at nonstandard primes $p\neq 2$}
Let $p$ be a nonstandard prime, we want to look at the $p$-adic density $\alpha_p(S)$. Let us assume that over $\Q_p$, $S$ has the Jordan decomposition $$S = \sum_{k=-\infty}^\infty p^{k} f_k,$$
where $f_k$ is the Jordan block with determinant coprime to $p$. We note that $f_k = 0$ for all but finitely many values of $k$. Let us denote the determinant of $f_k$ as $\Delta_k$ and the dimension of the Jordan block as $n_k$. Then for $p\neq 2$, the formula for the $p$-adic density is given by Conway and Sloane in \cite{ConwaySloane} as the following: 

\begin{equation}\label{eq: p-adic density formula}
    \alpha_p(S)^{-1} = 2 p^{-\frac{s(n+1)}{2}} \prod_{k} M_p(f_k) \prod_{k'>k} p^{\frac{(k'-k)n_k n_{k'}}{2}}
\end{equation}
where $p^s \| D$, the products over $k$ and $k'$ are over the indices such that $f_k\neq 0$, and $$M_p(f_k) = \begin{cases}
    \frac{1}{2} \prod_{i=1}^{\frac{n-1}{2}} (1-p^{-2i})^{-1}, & n\text{ is odd}, \\ 
    \frac{1}{2} (1- \legendre{(-1)^{n/2}\Delta_k}{p} p^{-n/2})^{-1} \cdot \prod_{i=1}^{\frac{n-2}{2}}(1-p^{-2i})^{-1}, & n\text{ is even.} 
\end{cases}$$

Notice that this formula still holds for $p$ a standard prime and hence are definitions are consistent.

\begin{lemma}\label{lem: p-adic density at p}
    Let $p$ be an odd prime and $n\geq 3$, then for $S$ a positive definite symmetric matrix of determinant $p$, we have that $$\alpha_p(S)^{-1} = \frac{p^{-1}}{2} \cdot \alpha_p^{\std}(S)^{-1} \cdot \left(1+O(p^{-2})\right).$$
\end{lemma}
\begin{proof}
    Since $p$ is an odd prime, we have that the only nontrivial Jordan blocks are $f_0$ and $f_1$. Moreover, $f_0$ has dimension $n_0 = n-1$ and $f_1$ has dimension $n_1 = 1$. Thus, we get that $$\alpha_p(S)^{-1} = 2p^{-1} M_p(f_0)M_p(f_1).$$
    Evaluating $M_p(f_0)$ and $M_p(f_1)$, we get that $$\alpha_p(S)^{-1} = \frac{p^{-1}}{2} \cdot \alpha_p^{\std}(S)^{-1} \cdot \left(1+O(p^{-2})\right).$$
\end{proof}

\subsection{The $2$-adic density}
For the $2$-adic density, there is a more complicated formula than (\ref{eq: p-adic density formula}) that takes into account the specifics of the Jordan decomposition. 
\begin{defn}
    A Jordan block $f_k$ is said to be \textbf{Type I} if $f_k$ represents an odd number in $\Z_2$. In this case, $f_k$ can be taken to be diagonal. Otherwise, $f_k$ is said to be \textbf{Type II} and will be a direct sum of matrices of the form $$\left\{\begin{pmatrix}
        2w& u\\ u&2v
    \end{pmatrix}: w,u,v\in \Z_2\right\}. $$ We note that Type II Jordan blocks can only occur if $n$ is even. 
\end{defn}
In \cite{ConwaySloane}, Conway and Sloane introduce the notion of love forms, which are $0$-dimensional Jordan blocks that still contribute to the formula. 

\begin{defn}
    A Jordan block $f_k$ with $n_k=0$ is called a \textbf{love form}. These forms are considered Type II forms.  
\end{defn}
\begin{remark}
    We note that in some sense love forms are nonexistent, as they are zero-dimensional Jordan blocks. Nonetheless, they are a useful trick and convention to simplify the exposition around the Siegel mass formula at $p=2$.
\end{remark}

If $f_k$ is a love form and $f_{k+1}$ or $f_{k-1}$ is a Type I Jordan block, then there is a contribution of $M_2(f_k)=\frac{1}{2}$; such a form $f_k$ is called a \textbf{bound love form}. Otherwise, $f_k$ is called a \textbf{free love form}, and $M_2(f_k) = 1$. With this context, we can introduce the modification of (\ref{eq: p-adic density formula}) at the place $p=2$: 
\begin{equation}\label{eq: 2-adic density}
    \alpha_p(S)^{-1} = 2 \cdot 2^{-\frac{s(n+1)}{2}} \prod_{k} M_2(f_k) \prod_{k'>k} 2^{\frac{(k'-k)n_k n_{k'}}{2}} \times 2^{n(I,I)-n(II)},
\end{equation}
where $2^s \| D$, $n(I,I)$ is the number of adjacent Type I Jordan blocks, and $n(II)$ is the sum of the dimensions of all of the Jordan blocks of Type II. 

In our case, since $D$ is some odd prime, the only relevant forms are $f_0$ and the love forms. If $f_0$ is a nontrivial Jordan block, the value of $M_2(f_0)$ depends on a$\bmod 8$ invariant. 

\begin{defn}
     Let $S$ denote some positive definite integral matrix of discriminant $D$. If $f_0$ is Type I, then it is diagonalizable and we take $E_i$ to be the number of elements on the diagonal that are congruent to $i \bmod 4$. We define the octane $\calO_S$ to be: $$\calO_S := \begin{cases}
        E_1 - E_3, & f_1\text{ is Type I},\\
        4-4\legendre{D}{2}, & f_1\text{ is Type II}.
    \end{cases}$$
The octane is a$\bmod8$ invariant.
\end{defn}

Now the definition of $M_2(f_0)$ is given by the following: 
\begin{equation}\label{eq: M2f1}
    M_2(f_0) = \frac{1}{2}\times\begin{cases}
       \prod_{i=1}^{\frac{n-3}{2}} (1-2^{-2i})^{-1} \cdot (1-2^{\frac{1-n}{2}})^{-1}, & n\text{ is odd, and $\calO_S\equiv \pm 1 \bmod 8$,} \\
       \prod_{i=1}^{\frac{n-3}{2}}(1-2^{-2i})^{-1} \cdot (1+2^{\frac{1-n}{2}})^{-1}, & n\text{ is odd, and $\calO_S\equiv \pm 3 \bmod 8$,}\\
       \prod_{i=1}^{\frac{n-2}{2}}(1-2^{-2i})^{-1}, & n\text{ is even, $f_0$ is Type I, and $\calO_S \equiv \pm 2 \bmod 8$},\\
       \prod_{i=1}^{\frac{n-4}{2}}(1-2^{-2i})^{-1} \cdot (1-2^{1-\frac{n}{2}})^{-1}, & n\text{ is even, $f_0$ is Type I, and $\calO_S\equiv 0 \bmod 8$}, \\
       \prod_{i=1}^{\frac{n-4}{2}}(1-2^{-2i})^{-1} \cdot (1+2^{1-\frac{n}{2}})^{-1}, & n\text{ is even, $f_0$ is Type I, and $\calO_S\equiv 4 \bmod 8$},\\ 
       \prod_{i=1}^{\frac{n-2}{2}}(1-2^{-2i})^{-1}\cdot (1-2^{-\frac{n}{2}})^{-1}, & n\text{ is even, $f_0$ is Type II, and $\calO_S \equiv 0 \bmod 8$}, \\
       \prod_{i=1}^{\frac{n-2}{2}} (1-2^{-2i})^{-1} \cdot (1+2^{-\frac{n}{2}})^{-1}, & n\text{ is even, $f_0$ is Type II, and $\calO_S\equiv 4 \bmod 8$}.
    \end{cases}
\end{equation}

Let us evaluate the $2$-density for when $D$ is an odd prime. 

\begin{lemma}\label{lem: 2-adic density at p}
    Let $D$ be an odd prime and $n\geq 3$. Then for $S$ a positive definite symmetric matrix of determinant $D$, we have that if $n$ is odd, then $$\alpha_2(S)^{-1} =\frac{1}{4} \prod_{i=1}^{\frac{n-3}{2}}(1-2^{-2i})^{-1} \times \begin{cases}
        (1-2^{\frac{1-n}{2}})^{-1}, & \calO_S \equiv \pm 1 \bmod 8, \\
        (1+2^{\frac{1-n}{2}})^{-1}, & \calO_S \equiv \pm 3 \bmod 8.
    \end{cases}$$
    If $n$ is even, then we have that 
    $$\alpha_2(S)^{-1} = \begin{cases}
            \frac{1}{4}\prod_{i=1}^{\frac{n-2}{2}}(1-2^{-2i})^{-1}, & n\text{ is even, $f_0$ is Type I, and $\calO_S \equiv \pm 2 \bmod 8$},\\
       \frac{1}{4}\prod_{i=1}^{\frac{n-4}{2}}(1-2^{-2i})^{-1} \cdot (1-2^{1-\frac{n}{2}})^{-1}, & n\text{ is even, $f_0$ is Type I, and $\calO_S\equiv 0 \bmod 8$}, \\
       \frac{1}{4}\prod_{i=1}^{\frac{n-4}{2}}(1-2^{-2i})^{-1} \cdot (1+2^{1-\frac{n}{2}})^{-1}, & n\text{ is even, $f_0$ is Type I, and $\calO_S\equiv 4 \bmod 8$},\\ 
       2^{-n} \prod_{i=1}^{\frac{n-2}{2}}(1-2^{-2i})^{-1}\cdot (1-2^{-\frac{n}{2}})^{-1}, & n\text{ is even, $f_0$ is Type II, and $\calO_S \equiv 0 \bmod 8$}, \\
       2^{-n}\prod_{i=1}^{\frac{n-2}{2}} (1-2^{-2i})^{-1} \cdot (1+2^{-\frac{n}{2}})^{-1}, & n\text{ is even, $f_0$ is Type II, and $\calO_S\equiv 4 \bmod 8$}.
        \end{cases}$$
\end{lemma}
\begin{proof}
    We will prove this by looking at the various cases. Since $D$ is an odd prime, we know the only nontrivial Jordan block occurs for $f_0$. First, we assume $n$ is odd. Then, we have that 
    \begin{align*}\alpha_2(S)^{-1} &= 2 M_2(f_{-1})M_2(f_0)M_2(f_1) \\
    &= \frac{1}{4} \prod_{i=1}^{\frac{n-3}{2}}(1-2^{-2i})^{-1} \times \begin{cases}
        (1-2^{\frac{1-n}{2}})^{-1}, & \calO_S \equiv \pm 1 \bmod 8, \\
        (1+2^{\frac{1-n}{2}})^{-1}, & \calO_S \equiv \pm 3 \bmod 8.
    \end{cases}\end{align*}
    Here we have inputted the contributions from the love forms $f_{-1}$ and $f_{1}$. 
    If $n$ is even, we have that 
    \begin{align*}
        \alpha_2(S)^{-1} &= 2M_2(f_{-1})M_2(f_0)M_2(f_1) \cdot 2^{n(I,I)-n(II)}\\
        &= \begin{cases}
            \frac{1}{4}\prod_{i=1}^{\frac{n-2}{2}}(1-2^{-2i})^{-1}, & n\text{ is even, $f_0$ is Type I, and $\calO_S \equiv \pm 2 \bmod 8$},\\
       \frac{1}{4}\prod_{i=1}^{\frac{n-4}{2}}(1-2^{-2i})^{-1} \cdot (1-2^{1-\frac{n}{2}})^{-1}, & n\text{ is even, $f_0$ is Type I, and $\calO_S\equiv 0 \bmod 8$}, \\
       \frac{1}{4}\prod_{i=1}^{\frac{n-4}{2}}(1-2^{-2i})^{-1} \cdot (1+2^{1-\frac{n}{2}})^{-1}, & n\text{ is even, $f_0$ is Type I, and $\calO_S\equiv 4 \bmod 8$},\\ 
       2^{-n} \prod_{i=1}^{\frac{n-2}{2}}(1-2^{-2i})^{-1}\cdot (1-2^{-\frac{n}{2}})^{-1}, & n\text{ is even, $f_0$ is Type II, and $\calO_S \equiv 0 \bmod 8$}, \\
       2^{-n}\prod_{i=1}^{\frac{n-2}{2}} (1-2^{-2i})^{-1} \cdot (1+2^{-\frac{n}{2}})^{-1}, & n\text{ is even, $f_0$ is Type II, and $\calO_S\equiv 4 \bmod 8$}.
        \end{cases}
    \end{align*}
    Here we have used that if $f_0$ is Type II, we have that $2^{n(I,I)-n(II)} = 2^{-n}$ and no contribution from $f_{-1}$ or $f_1$ since they are both free love forms. In the case, where $f_0$ is Type I, we have used that $2^{n(I,I)-n(II)} = 1$ and we have a contribution of $\frac{1}{2}$ from each of the two bound love forms $f_{-1}$ and $f_1$.
\end{proof}

\section{Symmetric matrices}\label{sec: sym}
Let $V = \Sym_n = \{X\in \Mat_n:\,X^T = X\}$ be a $\Z$-linear space. Unlike the case of $(\Mat_n,\det)$, the pair ($\Sym_n,\det$) is not a Hardy-Littlewood system; however, along the indefinite orbits it is a ``almost" a Hardy-Littlewood system. We will first treat these orbits.
\subsection{The indefinite orbit}\label{subsec: nondefinite orbit} Let us fix $p,q$ such that $p+q=n$ and $p,q\neq 0.$
For the orbits $\calO_{p,q}$ of symmetric matrices with signature $(p,q)$, we apply the following result of Oh in \cite{Oh}. Here we denote by $\mu_m$ the Tamagawa measure on the variety $V_m(\mathbb{A}):$ 
\begin{thm}[Oh \cite{Oh}, Example 4.3]
Let $n\geq 3$. 
Let $\Omega\subset \calO_{p,q}$ be a nice connected compact subset. Then we have that as $m\rightarrow\infty$, $$N(m,\Omega) = \#\{V_m(\Z) \cap \R^{+}\Omega\} = (1+o_\Omega(1))\int_{m^{\frac{1}{n}}\Omega\times \prod\limits_p V_m(\Z_p)}\delta(x)d\mu_m(x). $$
Here, the density function $\delta:V_m(\A) \rightarrow \R$ is constant on any adelic orbit $\mathcal{O}(\mathbb{A})\subset V_m(\A)$ and satisfies that for $x\in \mathcal{O}(\mathbb{A})$ we have: $$\delta(x)=\begin{cases}
    2, & \mathcal{O}(\mathbb{A}) \text{ contains a $\Q$-rational point}, \\
    0, & \text{otherwise.} \\
\end{cases}.$$
\begin{remark}
    As described in Example 6.3 of \cite{BR}, we can interpret $x\in V_m(\A)$ itself as a quadratic form and thus $$\delta(x) \neq 0 \iff \prod_{p} c_p(x_p) = 1,$$
where $c_p(x_p)$ are the Hasse-Minkowski invariants (see \cite{Cassels}) and the product ranges over both finite and infinite places. 
\end{remark}

\end{thm}
Restricting to the case of when $m$ is some odd prime $q$, we notice that for any prime $p\neq 2,q,$ there is a unique $\SL_n(\Z_p)-$orbit inside of $V_q(\Z_p).$ Furthermore, in this case for any $x\in V_q(\A)$, we have that $$c_p(x_p)=1.$$ Since the signature of $x_\infty\in \Omega$ is fixed, we have that $$c_\infty(x_\infty) = (-1)^{s(s-1)/2}=: c_\infty(\Omega),$$
where $s$ is the negative index of the signature of $x$. Hence, we can see that $$\delta(x) \neq 0 \iff c_2(x)c_q(x) = c_\infty(\Omega).$$
Recalling that the Tamagawa measure $\mu_m$ can be written as $$\mu_m = \prod_{\substack{p\\ \text{inc. }\infty}} \mu_{m,p},$$ we get the following:
\begin{corollary}\label{cor: N(q,Omega) sym}
Let $q$ be an odd prime. Then as $q\rightarrow\infty$, we have that 
\begin{multline*}N(q,\Omega) = (1+o_\Omega(1)) 2 q^{\frac{n-1}{2}}\mu_{\infty}(\Omega) \cdot \left(\prod_{p\neq 2,q}\mu_{q,p}(V_q(\Z_p))\right)\cdot\left(\sum_{\substack{\mathcal{G}_2,\calG_q\\ c_2(\calG_2)c_q(\calG_q)=c_\infty(\Omega)}}\mu_{q,2}(\calG_2(\Z_2))\mu_{q,q}(\calG_q(\Z_q))\right)\end{multline*} where for some $\SL_n(\Z_p)$-orbit $\calG_p\subset V_q(\Z_p)$, we have $$\mu_{p,q}(\calG_p(\Z_p)) = \lim_{k\rightarrow\infty}\frac{\#\calG_p(\Z/p^k\Z)}{p^{k(\frac{n(n+1)}{2}-1)}}.$$
    The sum over $\calG_2$ (resp. $\calG_q$) is over the orbits of $\SL_n(\Z_2)$ (resp. $\SL_n(\Z_q)$) inside of $V_q(\Z_2)$ (resp. $V_q(\Z_q)$).     
\end{corollary}

\begin{prop}\label{prop: evaluation of the constant symmetric}
For $n\geq 3$ odd, as the prime $q$ goes to $\infty$, we have that $$\left(\prod_{p\neq 2,q}\mu_{q,p}(V_q(\Z_p))\right)\cdot\left(\sum_{\substack{\mathcal{G}_2,\calG_q\\ c_2(\calG_2)c_q(\calG_q)=c_\infty(\Omega)}}\mu_{q,2}(\calG_2(\Z_2))\mu_{q,q}(\calG_q(\Z_q))\right) = \frac{1}{2}\prod_{\substack{3\leq j\leq n \\ j\text{ odd}}} \zeta(j)^{-1} (1+O(q^{-1})).$$
For $n\geq 4$ even, as the prime $q$ goes to $\infty$, we have that 
\begin{multline*}\prod_{p\neq 2,q}\mu_{q,p}(V_q(\Z_p))\cdot\left(\sum_{\substack{\mathcal{G}_2,\calG_q\\ c_2(\calG_2)c_q(\calG_q)=c_\infty(\Omega)}}\mu_{q,2}(\calG_2(\Z_2))\mu_{q,q}(\calG_q(\Z_q))\right) = \\ \frac{1}{2}\prod_{\substack{3\leq j\leq n\\ j\text{ odd}}} \zeta(j)^{-1} L(n/2, \legendre{(-1)^{n/2}q}{\cdot}) \cdot (1+O(q^{-1})) \\ \times\zeta(n)^{-1}  \begin{cases}
    1, & q\not\equiv (-1)^{n/2} \bmod 4, \\
    1+\frac{2^{1-n}}{1-2^{-n/2}}, & q \equiv (-1)^{n/2}\bmod 4, q\equiv \pm 1 \bmod 8,\\
    1-\frac{2^{1-n}}{1-2^{-n/2}}, & q\equiv (-1)^{n/2}\bmod 4, q\equiv \pm 3 \bmod 8.
\end{cases}\end{multline*}
\end{prop}
\begin{proof}
First, we remark that for any finite place $p$, and any $\SL_n(\Z_p)$-orbit $\calG$ of matrices inside $V_q(\Z_p)$, we have that $$\mu_{p,q}(\calG(\Z_p)) = \lim_{k\rightarrow\infty}\frac{\#\calG(\Z/p^k\Z)}{p^{k(\frac{n(n+1)}{2}-1)}},$$
where $\#\calG(\Z/p^k\Z)$ are the number of matrices in $V_q(\Z/p^k\Z)$ inside the orbit $\calG$. \\

We begin our analysis with the case of $n$ odd. First, we look at the standard primes: $p\neq 2,q.$ When dealing with these primes, the action of $\SL_n(\Z/p^k\Z)$ on $V_q(\Z/p^k\Z)$ given by $S\cdot g = g^T S g$ for $g\in \SL_n(\Z/p^k\Z)$ and $S\in V_q(\Z/p^k\Z)$ is transitive. Thus, we get that 
    $$\# V_q(\Z/p^k\Z) = \frac{\#\SL_n(\Z/p^k\Z)}{\#\{g\in \SL_n(\Z/p^k\Z): g^T S g= S\}}.$$
    For the numerator, it is not hard to see that $$\#\SL_n(\Z/p^k\Z) = p^{k(n^2-1)} \prod_{j=2}^n (1-p^{-j}).$$ On the other hand, for $k$ sufficiently large, $$\#\{X \in \GL_n(\Z/p^k\Z): X^T S X = S\} = 2\alpha_p(S) p^{k\frac{n(n-1)}{2}}.$$
    This in turn gives us that since $n$ is odd and $p\neq 2,q$, $$\#\{g\in \SL_n(\Z/p^k\Z):g^TSg=S\} = \alpha_p(S) p^{k\frac{n(n-1)}{2}}.$$
    Thus, we have that $$\lim_{k\rightarrow\infty}\frac{\#V_q(\Z/p^k\Z)}{p^{k(\frac{n(n+1)}{2}-1)}} = \alpha_p^{\std}(S)^{-1} \prod_{j=2}^n (1-p^{-j}) = \prod_{\substack{3\leq j\leq n\\ j \text{ odd}}} (1-p^{-j}).$$

    Next, we handle the case when $n\geq 3$ is odd and $p=q.$ In this case, the action of $\SL_n(\Z/q^k\Z)$ on $V_q(\Z/q^k\Z)$ is no longer transitive and we have two disjoint orbits. Let us call the two representatives $S_1^q$ and $S_2^q$. Like before, for $k$ sufficiently large, we can express the stabilizers in terms of the Siegel masses: $$\#\{X\in \GL_n(\Z/q^k\Z) : X^T S_i^q X = S_i^q\} = 2\alpha_q(S_i) q^{k\frac{n(n-1)}{2}}.$$
    We note that this indeed implies $$\#\{g\in \SL_n(\Z/q^k\Z): g^T S_i^q g = S_i^q\} = \frac{2\alpha_q(S_i^q) q^{k\frac{n(n-1)}{2}}}{\#\{x\bmod q^k: x^2q \equiv q \bmod q^k\}} = \alpha_q(S_i^q) q^{\frac{kn(n-1)}{2}-1}.$$
    Hence, for any of the orbits $\calG_q$: $$\lim_{k\rightarrow\infty} \frac{\#\mathcal{G}_q(\Z/q^k\Z)}{q^{k(\frac{n(n+1)}{2}-1)}} = q\prod_{j=2}^n(1-q^{-j}) \alpha_q(\mathcal{G}_q)^{-1}.$$
    Applying Lemma \ref{lem: p-adic density at p}, we can see that $$\lim_{k\rightarrow\infty} \frac{\#\mathcal{G}_q(\Z/q^k\Z)}{q^{k(\frac{n(n+1)}{2}-1)}} = \frac{1}{2} \prod_{\substack{3\leq j\leq n\\ j\text{ odd}}} (1-q^{-j}) (1+O(q^{-2})).$$

    Finally, we look at the case $p=2$. We have two orbits here generated by $S_1^2$ and $S_2^2$, where $\calO_{S_1} \equiv \pm 1 \bmod 8$ and $\calO_{S_2}\equiv \pm 3 \bmod 8$. We note that both $S_1^2$ and $S_2^2$ will appear in the final sum; in particular, for each of the $S_i^2$ there exists exactly one $\SL_n(\Z_q)$-orbit, represented by $S_j^q$, such that $$c_2(S_i^2)c_q(S_j^q) = c_\infty(\Omega).$$  Hence, we get that  
    $$\prod_{p\neq 2,q}\mu_{q,p}(V_q(\Z_p))\cdot\hspace{-0.75cm}\sum_{\substack{\mathcal{G}_2,\calG_q\\ c_2(\calG_2)c_q(\calG_q)=c_\infty(\Omega)}}\hspace{-0.75cm}\mu_{q,2}(\calG_2(\Z_2))\mu_{q,q}(\calG_q(\Z_q)) = \frac{1}{2}\prod_{\substack{3\leq j\leq n\\ j\text{ odd}}} \prod_{p\neq 2} (1-p^{-j})(1+O(q^{-2})) \cdot \sum_{\calG_2} \mu_{q,2}(\calG_2(\Z_2)),$$ where we are summing over the disjoint orbits inside of $V_q(\Z_2).$ \\
    
    In order to conclude the odd case of the Proposition, we are left with evaluating
    $$ \#V_q(\Z/2^k\Z)= \frac{\#\SL_n(\Z/2^k\Z)}{\#\{g\in \SL_n(\Z/2^k\Z): g^T S_1^2 g = S_1^2\}}+\frac{\#\SL_n(\Z/2^k\Z)}{\#\{g\in \SL_n(\Z/2^k\Z): g^T S_2^2 g = S_2^2\}}.$$
    We note again that $$\#\{X\in \GL_n(\Z/2^k\Z): X^T S_i^2 X = S_i^2\} = 2\alpha_2(S_i^2) 2^{\frac{kn(n-1)}{2}}.$$
    For $n$ odd and $k$ sufficiently large,  we will have that $$\#\{g\in \SL_n(\Z/2^k\Z):g^TS_i^2g = S_i^2\} = \frac{1}{4} \#\{X\in \GL_n(\Z/2^k\Z): g^TS_i^2 g = S_i^2\} = \frac{\alpha_2(S_i^2)}{2} \cdot 2^{\frac{kn(n-1)}{2}},$$ where we have used that the equation $x^2\equiv 1\bmod 2^k$ has $4$ solutions for $k\geq 3.$ \\
    
    Summing these together, we obtain \begin{align*}\sum_{\calG_2} \mu_{q,2}(\calG_2(\Z_2)) = \lim_{k\rightarrow\infty} \frac{\#V_q(\Z/2^k\Z)}{2^{k(\frac{n(n+1)}{2}-1)}} &= \prod_{j=2}^n (1-2^{-j})(2\alpha_2(S_1^2)^{-1} + 2\alpha_2(S_2^2)^{-1})\\
    &= \prod_{\substack{3\leq j\leq n\\ j\text{ odd}}}(1-2^{-j}) \cdot (1-2^{1-n})\cdot \left(\frac{1}{2}(1-2^{\frac{1-n}{2}})^{-1} + \frac{1}{2}(1+2^{\frac{1-n}{2}})^{-1}\right).\end{align*}
    Above, we applied Lemma \ref{lem: 2-adic density at p}. We evaluate this expression to derive: 
    $$\lim_{k\rightarrow\infty} \frac{\#V_q(\Z/2^k\Z)}{2^{k(\frac{n(n+1)}{2}-1)}} =\prod_{\substack{3\leq j\leq n\\ j\text{ odd}}}(1-2^{-j}).$$
    Combining all of the factors together, we have the odd case of Proposition \ref{prop: evaluation of the constant symmetric}. \\

    We now move on to the even case. The strategy will be the same as before, with the details being slightly more cumbersome due to the complexities of the possible Type II Jordan blocks showing up at the prime $2$. We start by handling the standard primes: $p \neq 2,q.$ Again, we consider the action of $\SL_n(\Z/p^k\Z)$ on $V_q(\Z/p^k\Z)$ given by $S\cdot g= g^T Sg$. Like before, the action is transitive and since $S$ is diagonalizable ($p$ is a standard prime), we can assume that $S$ is diagonal. We have $$\#V_q(\Z/p^k\Z)  = \frac{\#\SL_n(\Z/p^k\Z)}{\#\{g\in \SL_n(\Z/p^k\Z): g^TSg = S\}} = \frac{2\cdot \#\SL_n(\Z/p^k\Z)}{\#\{g\in \GL_n(\Z/p^k\Z) : g^TSg=S\}},$$
    where the second equality comes from the fact that since $S$ is diagonal, for any element of $\GL_n(\Z/p^k\Z)$ such that $g^TSg=S$ and $\det(g) = -1$, we have that $\text{diag}(-1,1,...,1)\cdot g \in \SL_n(\Z/p^k\Z)$ and $(\text{diag}(-1,1,...,1)\cdot g)^T S (\text{diag}(-1,1,...,1)\cdot g) = S$. Since $$\#\{X\in \GL_n(\Z/p^k\Z):X^TSX  = S\} = 2\alpha_p(S) p^{k\frac{n(n-1)}{2}} = 2 p^{k\frac{n(n-1)}{2}}\prod_{i=1}^{n/2-1} (1-p^{-2k}) \cdot (1-p^{-n/2}\legendre{(-1)^{n/2}q}{p}),$$
    we obtain 
    \begin{equation}\label{eq: local factor at standard primes even}\lim_{k\rightarrow\infty} \frac{\#V_q(\Z/p^k\Z)}{p^{k(\frac{n(n+1)}{2}-1)}} = \prod_{\substack{3\leq j\leq n\\ j\text{ odd}}}(1-p^{-j}) \cdot (1-p^{-n}) \left(1-p^{-n/2}\legendre{(-1)^{n/2} q}{p}\right)^{-1}.\end{equation}

    Next, we consider the case $p=q$ with $n$ even. There are now two orbits of the action of $\SL_n(\Z/q^k\Z)$ on $V_q(\Z/q^k\Z)$ and again we can take our two representatives to be diagonal matrices $S_1^q$ and $S_2^q$. We consider each term individually: For $k$ sufficiently large $$\#\{X\in \GL_n(\Z/q^k\Z) : X^T S_i^q X = S_i^q\} = 2\alpha_q(S_i^q) q^{k\frac{n(n-1)}{2}},$$
    which implies $$\#\{g\in \SL_n(\Z/q^k\Z): g^T S_i^q g = S_i^q\} = \frac{2\alpha_q(S_i^q) q^{k\frac{n(n-1)}{2}}}{\#\{x\bmod q^k: x^2q \equiv q \bmod q^k\}} = \alpha_q(S_i^q) q^{\frac{kn(n-1)}{2}-1}.$$
    Thus, we have that $$\lim_{k\rightarrow\infty} \frac{\#\mathcal{G}_q(\Z/q^k\Z)}{q^{k(\frac{n(n+1)}{2}-1)}} = q \prod_{j=2}^n (1-q^{-j})\alpha_q(\mathcal{G}_q)^{-1},$$ for $\calG_q$ an $\SL_n(\Z_q)$-orbit  inside of $V_q(\Z_q).$ From Lemma \ref{lem: p-adic density at p}, we know that $$\alpha_q(\mathcal{G}_q)^{-1}= \frac{1}{2q}\cdot \prod_{j=1}^{n/2-1} (1-q^{-2j})^{-1} \cdot \left(1+O(q^{-2})\right),$$
    concluding that 
    \begin{equation}\label{eq: local factor at prime p even}
        \lim_{k\rightarrow\infty} \frac{\#\mathcal{G}_q(\Z/q^k\Z)}{q^{k(\frac{n(n+1)}{2}-1)}} = \frac{1}{2}\prod_{\substack{3\leq j\leq n\\ j\text{ odd}}} (1-q^{-j}) \cdot (1-q^{-n}) \cdot (1+O(q^{-2})).
    \end{equation}
    
    We are left with the trickiest case to consider: $n$ is even and $p=2$. Again, for each $\SL_n(\Z_2)$-orbit there is exactly one $\SL_n(\Z_q)$-orbit, such that the product of the Hasse-Minkowski invariants is $c_\infty(\Omega)$. So, it suffices to calculate $$\sum_{\calG_2}\mu_{q,2}(\calG_2(\Z_2)) = \lim_{k\rightarrow\infty}\frac{\#V_q(\Z/2^k\Z)}{2^{k(\frac{n(n+1)}{2}-1)}}.$$
    To see this, we recall that $$\sum\limits_{\mathcal{G}_2}\#\mathcal{G}_2(\Z/2^k\Z)=\#V_q(\Z/2^k\Z) = \sum_{S_i} \frac{\#\SL_n(\Z/2^k\Z)}{\#\{X\in \SL_n(\Z/2^k\Z): X^TS_i X = S_i\}}.$$ As we explained earlier $$\#\{X \in \GL_n(\Z/2^k\Z): X^T S_i X = S_i\} = 4 \#\{X \in \SL_n(\Z/2^k\Z):X^TS_i X = S_i\}.$$
    Now, if $q \not\equiv (-1)^{n/2} \bmod 4$, then there are again two orbits -- both have $f_0$ being Type I, and either we have $\calO_{S_1} \equiv 0 \bmod 8, \calO_{S_2} \equiv 4 \bmod 8$ or $\calO_{S_1} \equiv 2 \bmod 8$ and $\calO_{S_2} \equiv -2 \bmod 8$. Thus, we have that $$\lim_{k\rightarrow\infty} \frac{\#V_q(\Z/2^k\Z)}{2^{k(\frac{n(n+1)}{2}-1)}} = \prod_{j=2}^n (1-2^{-j}) \left(2\alpha_2(S_1)^{-1} + 2\alpha_2(S_2)^{-1}\right).$$
    Applying Lemma \ref{lem: 2-adic density at p}, we see that if $\calO_{S_1}\equiv 0 \bmod 8$ and $\calO_{S_2}\equiv 4\bmod 8$, $$\alpha_2(S_1)^{-1} + \alpha_2(S_2)^{-1} = \frac{1}{4}\prod_{i=1}^{\frac{n-4}{2}}(1-2^{-2i})^{-1}\left((1-2^{1-n/2})^{-1} + (1+2^{1-n/2})^{-1}\right)=\frac{1}{2}\prod_{i=1}^{\frac{n-2}{2}}(1-2^{-2i})^{-1}.$$
    Similarly, if $\calO_{S_1}\equiv 2 \bmod 8$ and $\calO_{S_2}\equiv -2 \bmod 8$, we have 
    $$\alpha_2(S_1)^{-1} +  \alpha_2(S_2)^{-1} = \frac{1}{4}\prod_{i=1}^{\frac{n-4}{2}} (1-2^{-2i})^{-1} \left(2\cdot (1-2^{2-n})^{-1}\right) = \frac{1}{2}\prod_{i=1}^{\frac{n-2}{2}}(1-2^{-2i})^{-1}.$$
    Thus, in the case where $q\not\equiv (-1)^{n/2}\bmod 4$, we have that 
    \begin{equation}\label{eq: local density at 2 even nice q}
        \lim_{k\rightarrow\infty} \frac{\#V_q(\Z/2^k\Z)}{2^{k(\frac{n(n+1)}{2}-1)}} =\prod_{\substack{3\leq j\leq n\\ j\text{ odd}}}(1-2^{-j}) \cdot (1-2^{-n}).
    \end{equation}

    If $q \equiv (-1)^{n/2}\bmod 4$ and $q\equiv \pm 1 \bmod 8$, there are three orbits generated by $S_1,S_2,S_3$ -- two of them $S_1,S_2$ are such that $f_0$ is Type I and $\calO_{S_1}-\calO_{S_2} \equiv 4 \bmod 8$ and the final one $S_3$ is such that $f_0$ is Type II and $\calO_{S_3} \equiv 0 \bmod 8.$    
    We know that $$\lim_{k\rightarrow\infty} \frac{\#V_q(\Z/2^k\Z)}{2^{k(\frac{n(n+1)}{2}-1)}} = \prod_{j=2}^n (1-2^{-j}) \left(2\alpha_2(S_1)^{-1} + 2\alpha_2(S_2)^{-1} + 2\alpha_2(S_3)^{-1}\right),$$
    and we can compute \begin{align*}\alpha_2(S_1)^{-1} + \alpha_2(S_2)^{-1} + \alpha_2(S_3)^{-1} &= \frac{1}{2}\prod_{i=1}^{\frac{n-2}{2}}(1-2^{-2i})^{-1} + 2^{-n}\prod_{i=1}^{\frac{n-2}{2}}(1-2^{-2i})^{-1} \cdot (1-2^{-n/2})^{-1}\\
    &= \prod_{i=1}^{\frac{n-2}{2}} (1-2^{-2i})^{-1} \left(\frac{1}{2}+\frac{1}{2^n(1-2^{-n/2})}\right).\end{align*}
    Hence we obtain 
    \begin{equation}\label{eq: local density at 2 even q pm 1}
    \lim_{k\rightarrow\infty}\frac{\#V_q(\Z/2^k\Z)}{2^{k(\frac{n(n+1)}{2}-1)}} = \prod_{\substack{3\leq j\leq n\\ j\text{ odd}}}(1-2^{-j}) \cdot (1-2^{-n}) \cdot \left(1+\frac{2^{1-n}}{1-2^{-n/2}}\right).
    \end{equation}

    Our final case is when $q \equiv (-1)^{n/2} \bmod 4$ and $q\equiv \pm 3 \bmod 8.$ We have a similar situation here to the case before, except our three orbits are generated by $S_1,S_2,S_3$ -- where $S_{1}, S_2$ have $f_0$ Type I and $\calO_{S_1}-\calO_{S_2}\equiv 4 \bmod 8$, and $S_3$ has $f_0$ Type II and $\calO_{S_3}\equiv 4 \bmod 8.$ Following similar computations as above, we get that 
    \begin{equation}\label{eq: local density at 2 even q pm 3}
    \lim_{k\rightarrow\infty}\frac{\#V_q(\Z/2^k\Z)}{2^{k(\frac{n(n+1)}{2}-1)}} = \prod_{\substack{3\leq j\leq n\\ j\text{ odd}}}(1-2^{-j}) \cdot (1-2^{-n}) \cdot \left(1+\frac{2^{1-n}}{1+2^{-n/2}}\right).
    \end{equation}
    Combining (\ref{eq: local factor at standard primes even}), (\ref{eq: local factor at prime p even}), (\ref{eq: local density at 2 even nice q}), (\ref{eq: local density at 2 even q pm 1}), and (\ref{eq: local density at 2 even q pm 3}) together, we get the desired result. 
\end{proof}

Let us sum up the $N(p,\Omega)$ for $n\geq 3$ odd: applying Corollary \ref{cor: N(q,Omega) sym} and Proposition \ref{prop: evaluation of the constant symmetric}, we have that as $T\rightarrow\infty$ 
\begin{align*}\pi(T\Omega) &= \sum_{p\leq T^n} N(p,\Omega) \\
&=   2\mu_\infty(\Omega)(1+o_\Omega(1)) \sum_{p\leq T^n} p^{\frac{n-1}{2}} \cdot \frac{1}{2}\prod_{\substack{3\leq j\leq n\\ j\text{ odd}}}\zeta(j)^{-1} (1+O(p^{-1})) \\
&= \mu_\infty(\Omega) \prod_{\substack{3\leq j\leq n\\ j\text{ odd}}}\zeta(j)^{-1} (1+o_\Omega(1)) \sum_{p\leq T^n} p^{\frac{n-1}{2}} \\
&= \mu_\infty(\Omega) \prod_{\substack{3\leq j\leq n\\ j\text{ odd}}}\zeta(j)^{-1} \frac{2T^{\frac{n(n+1)}{2}}}{(n+1)\log(T^n)}(1+o_\Omega(1)).
\end{align*}
We remark that Lemma \ref{lem: the real part for cones PNT} tells us that $$\int_{T\Omega} \frac{1}{\log^+(\det(x))}dx = \mu_\infty(\Omega) \cdot \frac{T^{\dim(V)}}{\dim(V)\log(T)}(1+o(1)).$$
Since $\dim(V) = \frac{n(n+1)}{2},$ this gives us the following result. 

\begin{lemma}\label{lem: PNT cones symmetric n odd good orbit} For $n\geq 3$ odd, and $\Omega$ a compact subset in the orbit $\calO_{p,q}$ with $p,q\neq 0$,  we have that \begin{equation*}
    \pi(T\Omega) = \psi(p,q) \prod_{\substack{3\leq j\leq n\\ j\text{ odd}}} (1+o_\Omega(1)) \int_{T\Omega} \frac{1}{\log^+(\det(x))}dx,
\end{equation*}
where $$\psi(p,q) = \begin{cases}
    1, & q\text{ even},\\
    0, & \text{otherwise.}
\end{cases}$$
\end{lemma}

Finally, we end this section by summing up the $N(p,\Omega)$ for $n\geq 4$ even. We utilize that while Proposition \ref{prop: evaluation of the constant symmetric} has many cases, they are completely determined by the value of $p$ mod 8. So, for $T$ sufficiently large, we can write: 

\begin{align*}
    \pi(T\Omega) &= \sum_{\substack{p\leq T^n \\ p\not\equiv (-1)^{n/2}\bmod 4}} N(p,\Omega) + \sum_{\substack{p\leq T^n \\ p\equiv (-1)^{n/2}\bmod 4\\ p\equiv \pm 1 \bmod 8}} N(p,\Omega) + \sum_{\substack{p\leq T^n \\ p\equiv (-1)^{n/2}\bmod 4 \\ p\equiv \pm 3 \bmod 8}} N(p,\Omega) \\
    &= \frac{1}{2}\psi(p,q)\mu_\infty(\Omega)\zeta(n)^{-1} \prod_{\substack{3\leq j\leq n\\ j\text{ odd}}}\zeta(j)^{-1} (1+o_\Omega(1))\\ &\hspace{1cm}\times \Bigg(\sum_{\substack{p\leq T^n \\ p\not\equiv (-1)^{n/2}\bmod 4}} L(n/2, \chi_p) p^{\frac{n-1}{2}} + \sum_{\substack{p\leq T^n \\ p\equiv (-1)^{n/2}\bmod 4\\ p\equiv \pm 1 \bmod 8}} L(n/2,\chi_p) p^{\frac{n-1}{2}}\cdot (1+\frac{2^{1-n}}{1-2^{-n/2}}) \\
    &\hspace{4cm} + \sum_{\substack{p\leq T^n \\ p\equiv (-1)^{n/2}\bmod 4 \\ p\equiv \pm 3\bmod 8}} L(n/2,\chi_p) p^{\frac{n-1}{2}} (1+\frac{2^{1-n}}{1+2^{-n/2}})\Bigg),
\end{align*}
where $\chi_p = \legendre{(-1)^{n/2} p}{\cdot}.$ To evaluate the above sums, it suffices to determine an asymptotic for the following: 
\begin{lemma}\label{lem: summing the L function}
    For $n\geq 4$, we have that for $(a,8)=1$ as $T\rightarrow\infty$, $$\sum_{\substack{p\leq T^n \\ p\equiv a \bmod 8}} L(n/2,\chi_p) p^{\frac{n-1}{2}} = \frac{\zeta(n)}{4}\cdot \left(1-2^{-n}\right)\cdot \left(\frac{2T^{\frac{n(n+1)}{2}}}{n(n+1)\log(T)} (1+o(1)\right).$$
\end{lemma}
\begin{proof}
    Since $n\geq 4$, we can expand out the definition of $L(n/2,\chi_p)$: 
    \begin{align*}
        \sum_{\substack{p\leq T^n \\ p\equiv a \bmod 8}} p^{\frac{n-1}{2}} \sum_{\substack{m=1\\ (m,2p)=1}}^\infty \legendre{(-1)^{n/2} p}{m} \frac{1}{m^{n/2}} &= \sum_{\substack{m=1\\ (m,2)=1}}^\infty \frac{1}{m^{n/2}} \sum_{\substack{p\leq T^n \\ p\equiv a \bmod 8}} p^{\frac{n-1}{2}}\legendre{(-1)^{n/2} p }{m}.
    \end{align*}
    We note that since $m$ is odd, the character $\legendre{(-1)^{n/2}p}{m}$ is trivial if and only if $m$ is a square. In this case, we have the contribution to the sum is: 
    $$\sum_{\substack{\ell=1\\ (\ell,2)=1}}^\infty \frac{1}{\ell^n} \sum_{\substack{p\leq T^n \\ p\equiv a \bmod 8}}p^{\frac{n-1}{2}} = \frac{\zeta(n)}{4} \cdot (1-2^{-n})\cdot \left(\frac{2T^{\frac{n(n+1)}{2}}}{n(n+1)\log(T)} + O(T^{\frac{n(n+1)}{2}}\log(T)^{-101})\right).$$

    On the other hand, if $m$ is not a square and $m \geq \log(T)^{101}$, we have that since $n\geq 4$, $$\sum_{m\geq \log(T)^{101}} \frac{1}{m^{n/2}}\sum_{\substack{p\leq T^n \\ p\equiv a \bmod 8}}\legendre{(-1)^{n/2}p}{m} p^{\frac{n-1}{2}} \ll \frac{T^{\frac{n(n+1)}{2}}}{\log(T)} \sum_{m\geq \log(T)^{101}} \frac{1}{m^{n/2}}\ll T^{\frac{n(n+1)}{2}}\log(T)^{-102}.$$
    Finally, if $m$ is not a square, the character is nontrivial and we have that $$\sum_{\substack{p\leq T^n\\ p\equiv a\bmod 8}}\legendre{(-1)^{n/2}p}{m} p^{\frac{n-1}{2}}\ll T^{\frac{n(n+1)}{2}}\log(T)^{-101},$$ for $m\leq \log(T)^{101}$.
    Together, we get the lemma. 
\end{proof}
Plugging this into $\pi(T\Omega)$, we get:
\begin{align*}
    \pi(T\Omega) &= \psi(p,q)\mu_\infty(\Omega) \prod_{\substack{3\leq j\leq n\\ j\text{ odd}}}\zeta(j)^{-1} (1+o_\Omega(1)) \cdot \frac{2T^{\frac{n(n+1)}{2}}}{n(n+1)\log(T)} \\
    &\hspace{1cm} \times \Bigg(\frac{1}{2}(1-2^{-n})+\frac{1}{4}(1-2^{-n})(1+\frac{2^{1-n}}{1-2^{-n/2}}) + \frac{1}{4}(1-2^{-n})(1+\frac{2^{1-n}}{1+2^{-n/2}})\Bigg)\\
    &= \psi(p,q)\mu_\infty(\Omega) \prod_{\substack{3\leq j\leq n\\ j\text{ odd}}}\zeta(j)^{-1} \cdot \frac{2T^{\frac{n(n+1)}{2}}}{n(n+1)\log(T)} \cdot (1+o_\Omega(1)).
\end{align*}
Thus, we get the following lemma.
\begin{lemma}\label{lem: PNT cones symmetric n even good orbit} For $n\geq 3$ even, and $\Omega$ a nice compact subset in the orbit $\calO_{p,q}$ with $p,q\neq 0$,  we have that as $T\rightarrow\infty$, \begin{equation*}
    \pi(T\Omega) = \psi(p,q) \prod_{\substack{3\leq j\leq n\\ j\text{ odd}}} (1+o_\Omega(1)) \int_{T\Omega} \frac{1}{\log^+(\det(x))}dx.
\end{equation*}
\end{lemma}
\subsection{The positive-definite orbit}
For positive definite symmetric matrices, we can no longer apply Oh's result, since the stabilizer $\SO_n(\R)$ is a compact group. Instead, we appeal to an equidistribution result of Einsiedler, Margulis, Mohammadi, and Venkatesh \cite{EMMV}. We will proceed for the case of positive definite matrices $\calO_{n,0}$ (which produce positive primes), as the negative definite case is identical.

\subsubsection{Linnik equidistribution}
Let $\mathcal{M}=\PGL_n(\mathbb{Z})\backslash \PGL_n(\mathbb{R})/\PO_n(\mathbb{R})$ be the space of $\PGL_n(\mathbb{Z})$-equivalence classes of positive definite quadratic forms over $\mathbb{R}^n$. We equip $\mathcal{M}$ with the pushforward of the Haar measure on $\PGL_n(\mathbb{Z})\backslash P\GL_n(\mathbb{R})$. We now quote the main equidistribution result that we require from Einsiedler, Margulis, Mohammadi, and Venkatesh \cite{EMMV}: 

\begin{thm}[Einsiendler, Margulis, Mohammadi, Venkatesh, \cite{EMMV} Theorem 3.1]\label{thm: EMMV equidistribution of geni}
    Suppose $\{Q_i\}_{i=1}^\infty$ varies through any sequence of pairwise inequivalent, integral, positive definite quadratic forms. Then the genus of $Q_i$, considered as a subset of $\PGL_n(\Z)\backslash \PGL_n(\R) /\PO_n(\R)$, equidistributes as $i\rightarrow\infty.$
\end{thm}

Recall that two integral quadratic forms in the same genus have the same determinant. As a corollary we have the following:
\begin{corollary}
    Let $\mathcal{M}=\PGL_n(\mathbb{Z})\backslash \PGL_n(\mathbb{R})/\PO_n(\mathbb{R})$ be the space of positive definite quadratic forms of determinant $1$ up to $\Z$-equivalence, equipped with $\mu$, the pushforward of the Haar measure on $\PGL_n(\mathbb{Z})\backslash \PGL_n(\mathbb{R})$ normalized so that $\mu(\mathcal{M})=1$. Let $Q_m\subset \mathcal{M}$ denote the family of integral quadratic forms of determinant equal to $m$. If $\mu_m$ denotes the normalized Dirac measure over the family $Q_m$, then as $m\rightarrow\infty$:
    $$\mu_m\xrightarrow{\text{$w^*$}}\mu.$$
\end{corollary} 

\begin{proof}

    We note that for a fixed value of $m$, since $Q_m$ contains all integral quadratic forms of determinant $m$, we can write it as $$Q_m = \cup_{\calG_i} Q_{m,\calG_i},$$
    where $\calG_i$ runs through the set of genera with determinant $m$ and $Q_{m,\calG_i}$ is a set containing all of the inequivalent integral forms inside of $\calG_i$.  Applying Theorem \ref{thm: EMMV equidistribution of geni} to $Q_{m,\calG_i}$, we have that these genera will equidistribute as $m\rightarrow\infty$. In otherwords, if $\mu_{m,\calG_i}$ denotes the Dirac measure on $Q_{m,\calG_i}$, then we have that as $m\rightarrow\infty$, $$\frac{\mu_{m,\calG_i}}{\#Q_{m,\calG_i}}\xrightarrow{w^*}\mu.$$
    This then implies that $\frac{\mu_m}{\#Q_m}\xrightarrow{w^*}\mu$ as desired. 
\end{proof}

We can naturally identify $\PGL_n(\mathbb{Z})\backslash \PGL_n(\mathbb{R})/\PO_n(\mathbb{R})$ with the space $\SL_n(\mathbb{Z})\backslash \SL_n(\mathbb{R})/\SO_n(\mathbb{R})$. Under this identification we have that for any compact $\Omega_0 \subset \SL_n(\Z)\backslash\SL_n(\R)/\SO_n(\R)$,

\begin{equation*}\lim_{m\rightarrow \infty} \frac{\#\{A\in \Omega_0: m^{\frac{1}{n}}A \in \Mat_n(\mathbb{Z}) \}}{h_n(m)}\rightarrow \mu (\Omega_0),\end{equation*} where $h_n(m)$ is the class number of $\SL_n(\mathbb{Z})$-conjugacy orbits of integral positive definite symmetric matrices of determinant $m$. 

Next, we note that for any compact $\Omega\subset \SL_n(\R)/\SO_n(\R)$, we can decompose $$\Omega = \cup_{i=1}^k \Omega_i,$$
where $\Omega_i$ are distinct and each contained in different copies of the fundamental domain for the action of $\SL_n(\Z)$ on $\SL_n(\R)/\SO_n(\R).$ Then if we define a lift $\overline{\mu}$ of $\mu$ as $$\overline{\mu}(\Omega) = \sum_{i=1}^k \mu(\Omega_i),$$ we have that 
\begin{equation}\label{eq: linnik pos def}
    \lim_{m\rightarrow\infty} \frac{\#\{A\in \Omega: m^{1/n} A \in \Mat_n(\Z)\}}{h_n(m)}\rightarrow \overline{\mu}(\Omega).
\end{equation}
We will abuse notation and denote $\overline{\mu}$ as $\mu$ from now on. 
\subsubsection{Computing the class number}
Now, we will compute the class number $h_n(p)$ as $p$ ranges over primes. We note that for $n$ odd, this is computed by Kitaoka in \cite{KY}. We start at the same point, with a result of Kitaoka relating the class number to Siegel masses:
\begin{thm}[\cite{KY}, Theorem 2]\label{thm: Kitaoka}
    Let $n\geq 2$. Then as $m\rightarrow\infty$, we have that $$h_n(m) = (1+o(1)) \cdot \begin{cases}
        2\sum_{\calG_i} M(\calG_i), &n \text{ odd},\\
        4 \sum_{\calG_i} M(\calG_i), & n\text{ even}.
    \end{cases}$$
    Here the sum ranges over the genera of positive definite quadratic forms of determinant $m$. 
\end{thm}
Using the above, we derive the following estimate for $h_n(q)$ for $q$ prime: 
\begin{prop}\label{prop: class number}
    Let $n\geq 3$. Then for $n$ odd, as $q\rightarrow\infty$ for $q$ prime, we have that:
    $$h_n(q) = (1+o(1))\frac{q^{(n-1)/2}}{\pi^{n(n+1)/4}} \prod_{k=1}^n \Gamma(k/2) \prod_{k=1}^{\frac{n-1}{2}} \zeta(2k).$$
    For $n\geq 4$ even, as the prime $q$ goes to $\infty$, we have that 
    \begin{multline*}h_n(q) = (1+o(1)) \frac{q^{\frac{n-1}{2}}}{\pi^{\frac{n(n+1)}{4}}}\prod_{k=1}^n \Gamma(k/2) \prod_{k=1}^{\frac{n-1}{2}} \zeta(2k) L(n/2,\legendre{(-1)^{n/2}q}{\cdot} \\\times \begin{cases}
        2, & q\not\equiv (-1)^{n/2} \bmod 4\\
        2 + \frac{2^{-n+2}}{1-2^{-n/2}}, & q\equiv (-1)^{n/2} \bmod 4, q\equiv \pm 1 \bmod 8\\
        2 + \frac{2^{-n+2}}{1+2^{-n/2}}, & q \equiv (-1)^{n/2}\bmod 4, q\equiv \pm 3 \bmod 8.
    \end{cases}\end{multline*}
\end{prop}
\begin{remark}
    We note that in the case of $n$ odd, this asymptotic was previously derived by Kitaoka in \cite{KY}.
\end{remark}
\begin{proof}
    First, we handle the case that $n$ is odd. Using the Siegel mass formula and (\ref{eq: standard p mass}), we know that for any genus $\calG$, we have that $$M(\calG) = 2 \frac{q^{\frac{n+1}{2}}}{\pi^{\frac{n(n+1)}{4}}}\prod_{k=1}^n \Gamma(k/2) \prod_{p\neq 2,q} \left(\prod_{k=1}^{\frac{n-1}{2}} (1-p^{-2k})^{-1}\right)\cdot \alpha_q(\calG)^{-1}\alpha_2(\calG)^{-1}. $$
    Now, Lemma \ref{lem: p-adic density at p} tells us that $$\alpha_q(\calG)^{-1} = \frac{1}{2q} \prod_{k=1}^{\frac{n-1}{2}} (1-q^{-2k})^{-1}\cdot (1+O(q^{-1})).$$
    On the other hand, Lemma \ref{lem: 2-adic density at p} tells us that $$\alpha_2(\calG)^{-1} = \frac{1}{4}\prod_{k=1}^{\frac{n-3}{2}} (1-2^{-2k})^{-1} \times \begin{cases}
        (1-2^{\frac{1-n}{2}})^{-1}, & \calO_{\calG}\equiv \pm 1 \bmod 8\\
        (1+2^{\frac{1-n}{2}})^{-1}, & \calO_{\calG} \equiv \pm 3 \bmod 8.
    \end{cases}$$
     We know that there are two $\SL_n(\Z_2)$-orbits $\calG_1$ and $\calG_2$ of determinant $q$, one satisfying $\calO_{\calG_1}\equiv \pm 1 \bmod 8$ and the other satisfying $\calO_{\calG_2}\equiv \pm 3\bmod 8. $ This follows from the fact that there are two classes over $\Z_q$ and two classes over $\Z_2$, and the quadratic reciprocity relation (see Cassels \cite{Cassels}, Chapter 9). Thus, we have that 
     \begin{align*}
         h_n(q) &= (1+o(1))\cdot \frac{q^{(n-1)/2}}{2\pi^{n(n+1)/4}}\prod_{k=1}^n \Gamma(k/2) \prod_{k=1}^{\frac{n-1}{2}}\zeta(2k)\cdot (1-2^{1-n}) \left(\frac{1}{1-2^{\frac{1-n}{2}}}+\frac{1}{1+2^{\frac{1-n}{2}}}\right) \\
         &= (1+o(1))\cdot \frac{q^{(n-1)/2}}{\pi^{n(n+1)/4}} \prod_{k=1}^n \Gamma(k/2) \prod_{k=1}^{\frac{n-1}{2}} \zeta(2k).
     \end{align*}

    Next, we handle the more involved case where $n$ is even. First, we consider when $q\not\equiv (-1)^{n/2}\bmod 4.$ Again applying the Siegel mass formula, (\ref{eq: standard p mass}), Lemma \ref{lem: p-adic density at p}, and Lemma \ref{lem: 2-adic density at p}, we have that for $\calG$ of Type I:
    \begin{multline*}M(\calG) = \frac{q^{\frac{n-1}{2}}}{\pi^{n(n+1)/4}}\prod_{k=1}^n \Gamma(k/2) \prod_{k=1}^{\frac{n}{2}-1} \zeta(2k) L(n/2,\legendre{(-1)^{n/2}q}{\cdot}) \\ \times\frac{1}{4}(1-2^{2-n}) \cdot \begin{cases}
        (1-2^{1-\frac{n}{2}})^{-1}, & \calO_{\calG} \equiv 0 \bmod 8 \\
        (1+2^{1-\frac{n}{2}})^{-1}, & \calO_{\calG} \equiv 4 \bmod 8 \\ (1-2^{2-n})^{-1}, & \calO_{\calG} \equiv \pm 2 \bmod 8.
    \end{cases}\end{multline*}
    Again, we have two $\SL_n(\Z_2)$-orbits $\calG_1$ and $\calG_2$ that satisfy $\calO_{\calG_1}-\calO_{\calG_2}\equiv 4\bmod 8$. Hence we have that either
    \begin{multline*}
        h_n(q) = (1+o(1)) \cdot \frac{q^{(n-1)/2}}{\pi^{n(n+1)/4}}\prod_{k=1}^n \Gamma(k/2) \prod_{k=1}^{n/2-1}\zeta(2k)L(n/2, \legendre{(-1)^{n/2}q}{\cdot}) \\ \times (1-2^{2-n}) \left(\frac{1}{1-2^{1-n/2}}+ \frac{1}{1+2^{1-n/2}}\right)
    \end{multline*}
    or 
\begin{multline*}
        h_n(q) = (1+o(1)) \cdot \frac{q^{(n-1)/2}}{\pi^{n(n+1)/4}}\prod_{k=1}^n \Gamma(k/2) \prod_{k=1}^{n/2-1}\zeta(2k)L(n/2, \legendre{(-1)^{n/2}q}{\cdot}) \\ \times (1-2^{2-n})  \left(\frac{1}{1-2^{2-n}}+ \frac{1}{1-2^{2-n}}\right).
    \end{multline*}
    In both cases:
    $$h_n(q)= (2+o(1)) \cdot \frac{q^{(n-1)/2}}{\pi^{n(n+1)/4}}\prod_{k=1}^n \Gamma(k/2) \prod_{k=1}^{n/2-1}\zeta(2k)L(n/2, \legendre{(-1)^{n/2}q}{\cdot}).$$

Second, we consider when $q\equiv (-1)^{n/2}\bmod 4$ and $q \equiv \pm 1 \bmod 8.$ Now we will have three $\SL_n(\Z_2)$-orbits -- $\calG_1,\calG_2,\calG_3$ where $\calG_{1,2}$ have $f_0$ Type I and $\calO_{\calG_1}-\calO_{\calG_2}\equiv 4 \bmod 8$ and $\calG_3$ has $f_0$ Type II with $\calO_{\calG_3}\equiv 0 \bmod 8.$ Applying the Siegel mass formula, (\ref{eq: standard p mass}), Lemma \ref{lem: p-adic density at p}, and Lemma \ref{lem: 2-adic density at p}, we have an extra contribution from $\calG_3$: 
$$M(\calG_3) = \frac{q^{(n-1)/2}}{\pi^{n(n+1)/4}} \prod_{k=1}^n\Gamma(k/2) \prod_{k=1}^{\frac{n}{2}-1} \zeta(2k) L(n/2, \legendre{(-1)^{n/2}q}{\cdot}) 2^{-n} (1-2^{-n/2})^{-1}.$$
Combining this with the calculations above, we see that $$h_n(q) = (1+o(1)) \cdot (2 + \frac{2^{-n+2}}{1-2^{-n/2}})\cdot \frac{q^{(n-1)/2}}{\pi^{n(n+1)/4}}\prod_{k=1}^n \Gamma(k/2) \prod_{k=1}^{n/2-1}\zeta(2k)L(n/2, \legendre{(-1)^{n/2}q}{\cdot}).$$

Our final case is when $q\equiv (-1)^{n/2}\bmod 4$ and $q\equiv \pm 3 \bmod 8.$ We will have three $\SL_n(\Z_2)$-orbits -- $\calG_1,\calG_2,\calG_3$ where $\calG_{1,2}$ have $f_0$ Type I and $\calO_{\calG_1}-\calO_{\calG_2}\equiv 4 \bmod 8$ and $\calG_3$ has $f_0$ Type II with $\calO_{\calG_3}\equiv 4 \bmod 8.$ 
Applying the Siegel mass formula, (\ref{eq: standard p mass}), Lemma \ref{lem: p-adic density at p}, and Lemma \ref{lem: 2-adic density at p}, we have that the extra contribution from $\calG_3$ is given by 
$$M(\calG_3) = \frac{q^{(n-1)/2}}{\pi^{n(n+1)/4}} \prod_{k=1}^n\Gamma(k/2) \prod_{k=1}^{\frac{n}{2}-1} \zeta(2k) L(n/2, \legendre{(-1)^{n/2}q}{\cdot}) 2^{-n} (1+2^{-n/2})^{-1}.$$
Hence, combining with the computations above, we get that
$$h_n(q) = (1+o(1)) \cdot (2 + \frac{2^{-n+2}}{1+2^{-n/2}})\cdot \frac{q^{(n-1)/2}}{\pi^{n(n+1)/4}}\prod_{k=1}^n \Gamma(k/2) \prod_{k=1}^{n/2-1}\zeta(2k)L(n/2, \legendre{(-1)^{n/2}p}{\cdot}).$$
\end{proof}
\subsubsection{Comparison of the measures}
Before completing this section, we need to compare the Haar measure $\mu$ from (\ref{eq: linnik pos def}) to the Bateman-Horn measure $\mu_\infty.$
Let us first place our measures on the same space. Let us notate the following map:
$$\psi: \SL_n(\R)/\SO_n(\R) \rightarrow \calO_{n,0}\subset \R^{\frac{n(n+1)}{2}}$$
$$X\mapsto  XX^T.$$
Notice that this map is bijective and hence we can identify these spaces.
Denote by $\mu^{SO_n}$ the Riemannian measure on $\SO_n(\R)$ used by Zhang in \cite{Zhang2015VolumesOO}. Under this measure, the volume of $\SO_n(\R)$ is computed in \cite{Zhang2015VolumesOO} as $$\mu^{SO_n}(\SO_n(\R))= 2^{n-1} \pi^{\frac{n(n+1)}{4}} \prod_{k=1}^n \Gamma(k/2)^{-1}.$$ Denote further by $\mu_{0}^{SO_n}$ the normalization of this measure: $$\mu_0^{\SO_n(\R)}(\SO_n(\R)):=2^{-n+1} \pi^{\frac{-n(n+1)}{4}} \prod_{k=1}^n \Gamma(k/2)\mu^{\SO_n(\R)}.$$
We will use $\mu^{\SO_n}$ and $\mu_0^{\SO_n}$ to relate the measures $\mu$ and $\mu_\infty.$
\begin{lemma}\label{lem: ratio of measures}
For any compact set $\Omega\subset \SL_n(\R)/\SO_n(\R)$, we have that if $n$ is odd, then $$\mu(\Omega) = \pi^{\frac{n(n+1)}{4}}\prod_{j=2}^n \zeta(j)^{-1} \prod_{k=1}^n \Gamma(k/2)^{-1}\mu_\infty(\Omega).$$
If $n$ is even, then 
$$\mu(\Omega) = \frac{1}{2} \pi^{\frac{n(n+1)}{4}}\prod_{j=2}^n \zeta(j)^{-1} \prod_{k=1}^n \Gamma(k/2)^{-1}\mu_\infty(\Omega).$$
\end{lemma}
\begin{proof}
    Consider the following map:
    $$\phi: \calO_{n,0} \times\SO_n(\R) \rightarrow \SL_n(\R)$$
    $$(S,Q)\mapsto  \sqrt{S}Q.$$ Since $S$ is an identification as a positive definite symmetric matrix and positive definite symmetric matrices have unique square roots, this map is well-defined. Additionally, the map $\phi$ is bijective and the composition $$\calO_{n,0} \times \SO_n(\R) \xrightarrow{\phi} \SL_n(\R) \xrightarrow{X\mapsto XX^T} \calO_{n,0}$$
    sends $(S,Q) \mapsto S.$
    
    Let $\mu_\infty^{\SL_n}$ be the Bateman-Horn measure on $\SL_n(\R)$ as computed in Lemma \ref{lem: the real part for cones PNT}. We remark that this is a Haar measure on $\SL_n(\R).$ We also know that $\phi_*(d\mu\times d\mu_0^{\SO_n(\R)})$ is a Haar measure on $\SL_n(\R)$ and the Haar measure on $\SL_n(\R)$ is unique up to scalar multiple. We wish to compute the scalar between these two measures; we do so by computing the volumes of the two measures on a fundamental domain of $\SL_2(\Z)\backslash \SL_n(\R)$.
    
    Let $\calF\subset \SL_n(\R)$ be a fundamental domain for $\SL_n(\Z)\backslash \SL_n(\R)/\SO_n(\R)$ and $\calF'\subset \SL_n(\R)$ be a fundamental domain for $\SL_n(\Z)\backslash\SL_n(\R)$. We remark that for almost all $A\in\SL_n(\R)$, we can compute that $$A^{-1}\SL_n(\Z)A\cap\SO_n(\R)=c_n =\begin{cases}
        1, & n\text{ odd},\\ 2, & n\text{ even.}
    \end{cases}$$ Hence, we have that
    $$\int_{\calF'}\phi_*(d\mu\times  d\mu_0^{\SO_n(\R)})=c_n^{-1}\mu(\calF)\mu_0^{\SO_n(\R)}(\SO_n(\R))=\begin{cases}
        1, & n \text{ odd},\\ 1/2, & n \text{ even.}
    \end{cases}$$
    Here we have recalled that $\mu$ was normalized such that $\mu(\calF)=1.$
    However, Appendix A of \cite{DRS} computes that $$\mu_\infty^{\SL_n}(\calF') = \mu_\infty^{\SL_n}(\SL_n(\Z)\backslash\SL_n(\R)) = \prod_{j=2}^n\zeta(j),$$ and so we get the relation
    \begin{align*}
        \phi_*(d\mu\times d\mu_0^{\SO_n(\R)})&=\frac{1}{c_n\prod_{j=2}^n\zeta(j)}d\mu_{\infty}^{\SL_n(\R)}.
    \end{align*}
    On the other hand, by Lemma 2.27 in \cite{Zhang2015VolumesOO} 
    \begin{align*}
    d\mu_{\infty}\wedge d\mu^{\SO_n(\R)}&=2^{n-1} d\mu_{\infty}^{\SL_n(\R)}. 
    \end{align*}
    
    Putting these two together, we achieve that \begin{multline*}
        \phi_*(d\mu\times d\mu_0^{\SO_n(\R)})=\frac{1}{c_n2^{n-1}\prod_{j=2}^n\zeta(j)} d\mu_{\infty}\wedge d\mu^{\SO_n(\R)}\\=2^{n-1} \pi^{\frac{n(n+1)}{4}} \prod_{k=1}^n \Gamma(k/2)^{-1}\frac{1}{c_n2^{n-1}\prod_{j=2}^n\zeta(j)} d\mu_{\infty}\wedge d\mu_0^{\SO_n(\R)},
    \end{multline*}
    and hence $$d\mu= c_n^{-1} \pi^{\frac{n(n+1)}{4}}\prod_{j=2}^n \zeta(j)^{-1} \prod_{k=1}^n \Gamma(k/2)^{-1}d\mu_\infty.$$
    This completes the proof. 
\end{proof}

\subsubsection{Proof of Theorem \ref{thm: PNT det Sym}}
Finally, we can combine all of these results together. First we consider when $n$ is odd. Taking together (\ref{eq: linnik pos def}), Proposition \ref{prop: class number}, and Lemma \ref{lem: ratio of measures}, we have that as $p\rightarrow\infty$, $$N(p,\Omega) = (1+o_\Omega(1)) \mu_\infty(\Omega)\prod_{\substack{3\leq j\leq n\\ j\text{ odd}}}\zeta(j)^{-1} p^{\frac{n-1}{2}}.$$ Summing over $p$, we have the following result. 
\begin{lemma}\label{lem: PNT symmetric cones bad orbit n odd}
    For $n\geq 3$ odd, and $\Omega\subset \calO_{n,0}$ a compact subset, we have that as $T\rightarrow\infty$, 
    $$\pi(T\Omega) = (1+o_\Omega(1)) \prod_{\substack{3\leq j \leq n\\ j\text{ odd}}} \zeta(j)^{-1}\int_{T\Omega} \frac{1}{\log^+(\det(x))}dx.$$
\end{lemma}

Next, we deal with the case when $n$ is even. We can see that (\ref{eq: linnik pos def}), Proposition \ref{prop: class number}, and Lemma \ref{lem: ratio of measures} combine to form: 
\begin{multline*}N(p,\Omega) = (1+o_\Omega(1))\mu_\infty(\Omega) \prod_{\substack{3\leq j\leq n\\ j\text{ odd}}}\zeta(j)^{-1} \cdot \zeta(n)^{-1} L(n/2,\legendre{(-1)^{n/2}p}{\cdot}) \\ \times \begin{cases}
    1, & p\not\equiv (-1)^{n/2}\bmod 4\\ 
    1+\frac{2^{-n+1}}{1-2^{-n/2}}, & p\equiv (-1)^{n/2}\bmod 4, p\equiv \pm 1 \bmod 8\\
    1+\frac{2^{-n+1}}{1+2^{-n/2}}, & p\equiv (-1)^{n/2}\bmod 4, p\equiv \pm 3 \bmod 8.
\end{cases}\end{multline*}
Now we apply Lemma \ref{lem: summing the L function} to see that 
\begin{align*}
    \pi(T\Omega) &= (1+o_\Omega(1)) \mu_\infty(\Omega) \frac{2T^{n(n+1)/2}}{n(n+1)\log(T)} \prod_{\substack{3\leq j\leq n\\ j\text{ odd}}}\zeta(j)^{-1} \left((1-2^{-n}) \cdot \left(1+\frac{2^{-n+1}}{4(1-2^{-n/2})}+ \frac{2^{-n+1}}{4(1+2^{-n/2})}\right)\right)\\
    &= (1+o_\Omega(1)) \mu_\infty(\Omega) \frac{2T^{n(n+1)/2}}{n(n+1)\log(T)} \prod_{\substack{3\leq j\leq n\\ j\text{ odd}}}\zeta(j)^{-1} \\
    &= (1+o_\Omega(1)) \prod_{\substack{3\leq j\leq n\\ j\text{ odd}}}\zeta(j)^{-1} \int_{T\Omega} \frac{1}{\log^+(\det(x))}dx.
\end{align*}
Here we have applied Lemma \ref{lem: the real part for cones PNT} for the final equality.
This completes the proof of the following: 
\begin{lemma}\label{lem: PNT symmetric cones bad orbit n even}
    For $n\geq 3$ even, and $\Omega\subset \calO_{n,0}$ a compact subset, we have that as $T\rightarrow\infty$, 
    $$\pi(T\Omega) = (1+o_\Omega(1)) \prod_{\substack{3\leq j \leq n\\ j\text{ odd}}} \zeta(j)^{-1}\int_{T\Omega} \frac{1}{\log^+(\det(x))}dx.$$
    
\end{lemma}
Together Theorem \ref{thm: cones to boxes}, Lemma \ref{lem: PNT cones symmetric n odd good orbit}, Lemma \ref{lem: PNT cones symmetric n even good orbit}, Lemma \ref{lem: PNT symmetric cones bad orbit n odd}, and Lemma \ref{lem: PNT symmetric cones bad orbit n even} give us Theorem \ref{thm: PNT det Sym}.\qed

\section{Final remarks}\label{sec; remarks}
In this last section, we comment on variant problems that our methods can handle. First, we note that we can instead count primes in arithmetic progressions; for any $A>0$, $q\leq \log(T)^A$, and $(a,q)=1$, we can derive an asymptotic for
$$\#\{A\in \Mat_n(\Z): \|A\|\leq T, \det(A)\equiv a \bmod q, \det(A) \text{ is prime}\}.$$ 
This follows from applying Dirichlet's theorem for primes in arithmetic progressions in \S\ref{sec: PNT cones} rather than the prime number theorem.  

Similarly, we can count almost-primes using the above techniques with a slightly more involved computation with the Siegel masses, for instance asymptotics for the following expression for any finite $k\in \mathbb{N}$: $$\#\{A\in \Mat_n(\Z): \|A\|\leq T, \det(A)\text{ is a product of $k$-primes}\}.$$
In comparison, for a general homogeneous polynomial $F(\bx)$ of degree $d$ in $n$ variables, it is known that using the lower bound sieve of Diamond and Halberstam, one can show that for a constant $C_k$ and for $k>2d\cdot \frac{n-1}{n}+C_k$, $$\#\{\bx\in \Z^n:\|\bx\|\leq T, F(\bx) \text{ is a product of at most $k$-primes}\} \gg_k \frac{T^n}{\log(T)}.$$

\section*{Acknowledgments}
The authors thank their advisor Peter Sarnak for suggesting the problem and his insights described in his unpublished letter \cite{Sarnak_2022}. The authors would also like to thank Anshul Adve, Valeriya Kovaleva, Hee Oh, Alina Ostafe, Igor Shparlinski, and Nina Zubrilina for feedback on earlier drafts of the paper.
The second author is funded by the National Science Foundation Graduate Research Fellowship Program under Grant No. DGE-2039656. Any opinions, findings, and conclusions or recommendations expressed in this material are those of the author(s) and do not necessarily reflect the views of the National Science Foundation.

\bibliographystyle{abbrv}
\bibliography{biblio}

\begin{thebibliography}{10}

\bibitem{BatemanHorn}
P.~T. Bateman and R.~A. Horn.
\newblock A heuristic asymptotic formula concerning the distribution of prime numbers.
\newblock {\em Math. Comp.}, 16:363--367, 1962.

\bibitem{BH}
A.~Borel and Harish-Chandra.
\newblock Arithmetic subgroups of algebraic groups.
\newblock {\em Ann. of Math. (2)}, 75:485--535, 1962.

\bibitem{BR}
M.~Borovoi and Z.~Rudnick.
\newblock Hardy-{L}ittlewood varieties and semisimple groups.
\newblock {\em Invent. Math.}, 119(1):37--66, 1995.

\bibitem{Borovoi}
M.~V. Borovoi.
\newblock The {H}asse principle for homogeneous spaces.
\newblock {\em J. Reine Angew. Math.}, 426:179--192, 1992.

\bibitem{BrudernWooley}
J.~Bruedern and T.~D. Wooley.
\newblock On waring's problem for larger powers, 2022.

\bibitem{Cassels}
J.~W.~S. Cassels.
\newblock {\em Rational quadratic forms}, volume~13 of {\em London Mathematical Society Monographs}.
\newblock Academic Press, Inc. [Harcourt Brace Jovanovich, Publishers], London-New York, 1978.

\bibitem{conrad}
K.~Conrad.
\newblock A multivariable hensel’s lemma.
\newblock {\em Lecture note available at http://kconrad. math. uconn. edu/blurbs}, 2020.

\bibitem{ConwaySloane}
J.~H. Conway and N.~J.~A. Sloane.
\newblock Low-dimensional lattices. {IV}. {T}he mass formula.
\newblock {\em Proc. Roy. Soc. London Ser. A}, 419(1857):259--286, 1988.

\bibitem{Daniel}
S.~Daniel.
\newblock On the divisor-sum problem for binary forms.
\newblock {\em J. Reine Angew. Math.}, 507:107--129, 1999.

\bibitem{DS}
K.~Destagnol and E.~Sofos.
\newblock Rational points and prime values of polynomials in moderately many variables.
\newblock {\em Bull. Sci. Math.}, 156:102794, 33, 2019.

\bibitem{Duke}
W.~Duke.
\newblock Hyperbolic distribution problems and half-integral weight {M}aass forms.
\newblock {\em Invent. Math.}, 92(1):73--90, 1988.

\bibitem{DRS}
W.~Duke, Z.~Rudnick, and P.~Sarnak.
\newblock Density of integer points on affine homogeneous varieties.
\newblock {\em Duke Math. J.}, 71(1):143--179, 1993.

\bibitem{EMMV}
M.~Einsiedler, G.~Margulis, A.~Mohammadi, and A.~Venkatesh.
\newblock Effective equidistribution and property ($\tau$).
\newblock {\em Journal of the American Mathematical Society}, 33(1):223--289, 2020.

\bibitem{EO}
A.~Eskin and H.~Oh.
\newblock Representations of integers by an invariant polynomial and unipotent flows.
\newblock {\em Duke Math. J.}, 135(3):481--506, 2006.

\bibitem{FI1}
J.~Friedlander and H.~Iwaniec.
\newblock Asymptotic sieve for primes.
\newblock {\em Ann. of Math. (2)}, 148(3):1041--1065, 1998.

\bibitem{FI2}
J.~Friedlander and H.~Iwaniec.
\newblock The polynomial {$X^2+Y^4$} captures its primes.
\newblock {\em Ann. of Math. (2)}, 148(3):945--1040, 1998.

\bibitem{OperaDeCribro}
J.~Friedlander and H.~Iwaniec.
\newblock {\em Opera de cribro}, volume~57 of {\em American Mathematical Society Colloquium Publications}.
\newblock American Mathematical Society, Providence, RI, 2010.

\bibitem{HB}
D.~R. Heath-Brown.
\newblock Primes represented by {$x^3+2y^3$}.
\newblock {\em Acta Math.}, 186(1):1--84, 2001.

\bibitem{HBMoroz}
D.~R. Heath-Brown and B.~Z. Moroz.
\newblock On the representation of primes by cubic polynomials in two variables.
\newblock {\em Proc. London Math. Soc. (3)}, 88(2):289--312, 2004.

\bibitem{Iwaniec}
H.~Iwaniec.
\newblock Fourier coefficients of modular forms of half-integral weight.
\newblock {\em Invent. Math.}, 87(2):385--401, 1987.

\bibitem{KY}
Y.~Kitaoka.
\newblock Two theorems on the class number of positive definite quadratic forms.
\newblock {\em Nagoya Mathematical Journal}, 51:79--89, 1973.

\bibitem{Kovaleva}
V.~Kovaleva.
\newblock On the distribution of equivalence classes of random symmetric {$p$}-adic matrices.
\newblock {\em Mathematika}, 69(4):903--933, 2023.

\bibitem{LangWeil}
S.~Lang and A.~Weil.
\newblock Number of points of varieties in finite fields.
\newblock {\em American Journal of Mathematics}, 76(4):819--827, 1954.

\bibitem{Linnik}
J.~V. Linnik and M.~Keane.
\newblock {\em Ergodic properties of algebraic fields}, volume~45.
\newblock Springer, 1968.

\bibitem{LinnikSkubenko}
Y.~V. Linnik and B.~F. Skubenko.
\newblock On the asymptotic behavior of integral matrices of third order.
\newblock In {\em Doklady Akademii Nauk}, volume 146, pages 1007--1008. Russian Academy of Sciences, 1962.

\bibitem{MacWilliams}
J.~MacWilliams.
\newblock Orthogonal matrices over finite fields.
\newblock {\em Amer. Math. Monthly}, 76:152--164, 1969.

\bibitem{MaynardNormForm}
J.~Maynard.
\newblock Primes represented by incomplete norm forms.
\newblock {\em Forum Math. Pi}, 8:e3, 128, 2020.

\bibitem{Oh}
H.~Oh.
\newblock Hardy-{L}ittlewood system and representations of integers by an invariant polynomial.
\newblock {\em Geom. Funct. Anal.}, 14(4):791--809, 2004.

\bibitem{Ratner}
M.~Ratner.
\newblock On raghunathan's measure conjecture.
\newblock {\em Annals of Mathematics}, 134(3):545--607, 1991.

\bibitem{Sarnak}
P.~Sarnak.
\newblock Diophantine problems and linear groups.
\newblock In {\em Proceedings of the International Congress of Mathematicians}, volume~1, pages 459--471, 1990.

\bibitem{Sarnak_2022}
P.~Sarnak.
\newblock Letter to {S}arah and {V}aleriya, Jan 2022.

\bibitem{Siegel}
C.~L. Siegel.
\newblock {\"U}ber die {A}nalytische {T}heorie der {Q}uadratischen {F}ormen.
\newblock {\em Ann. Math.}, 36:527--606, 1935.

\bibitem{We}
A.~Weil.
\newblock {\em Adeles and algebraic groups}, volume~23 of {\em Progress in Mathematics}.
\newblock Birkh\"{a}user, Boston, MA, 1982.
\newblock With appendices by M. Demazure and Takashi Ono.

\bibitem{Zhang2015VolumesOO}
L.~Zhang.
\newblock Volumes of orthogonal groups and unitary groups.
\newblock {\em arXiv: Mathematical Physics}, 2015.

\end{thebibliography}

\end{document}